\documentclass[12pt]{amsart}

\usepackage[hmargin=2.5cm,bmargin=2.5cm,tmargin=3cm]{geometry}
\usepackage[usenames]{color}

\usepackage{amsmath,amssymb,amsthm,amsfonts,enumerate,url,mathbbol,mathrsfs,tikz,graphicx,pifont}
\usepackage{hyperref}

\usepackage[normalem]{ulem}
\usepackage[utf8]{inputenc}

\theoremstyle{plain}
\newtheorem{theorem}{Theorem}[section]
\newtheorem{lemma}[theorem]{Lemma}
\newtheorem{proposition}[theorem]{Proposition}
\newtheorem{corollary}[theorem]{Corollary}

\theoremstyle{definition}
\newtheorem{definition}[theorem]{Definition}
\newtheorem{example}[theorem]{Example}

\newtheorem{assumption}[theorem]{Assumption}
\newtheorem{remark}[theorem]{Remark}

\numberwithin{equation}{section}
\numberwithin{figure}{section}

\DeclareMathOperator{\esssup}{ess \, sup}

\DeclareMathOperator{\lvlsf}{\mathcal{G}}
\DeclareMathOperator{\diam}{diam}

\newcommand{\excset}{\mathfrak{S}}
\newcommand{\efd}{f_1}

\newcommand{\dx}{\; {\rm d}x}

\newcommand{\R}{\mathbb{R}}

\newcommand\mv{\mathsf{v}} 
\newcommand\mw{\mathsf{w}} 
\newcommand\me{\mathsf{e}} 
\newcommand\mV{\mathsf{V}} 
\newcommand\mE{\mathsf{E}} 

\newcommand{\mG}{\mathsf{G}} 
\newcommand{\mC}{\mathsf{C}} 
\newcommand{\mQ}{\mathsf{Q}}
\newcommand{\mK}{\mathsf{K}} 
\newcommand{\mW}{\mathsf{W}} 
\newcommand\Lfun{\mathscr{L}} 
\newcommand{\Graph}{{\mathcal G}} 
\newcommand{\Tree}{{\mathcal T}} 

\newcommand\mP{\mathsf{P}} 
\newcommand\mS{\mathsf{S}} 

\newcommand{\eigI}{\lambda} 
\newcommand{\eigIp}{\lambda^{(p)}} 
\newcommand{\eigO}{\mu} 
\newcommand{\RQ}[2]{\mathsf{R}_p(#1;#2)} 

\title[Edge connectivity and the spectral gap]{Edge connectivity and the spectral gap of combinatorial and quantum graphs} 

\subjclass[2010]{}


\author[G.~Berkolaiko]{Gregory Berkolaiko}
\author[J.B.~Kennedy]{James B. Kennedy}
\author[P.~Kurasov]{Pavel Kurasov}
\author[D.~Mugnolo]{Delio Mugnolo}

\address{Gregory Berkolaiko, Department of Mathematics, Texas A{\&}M University, College Station, TX 77843-3368, USA}
\email{gregory.berkolaiko@math.tamu.edu}

\address{James B.\ Kennedy, Institut f\"ur  Analysis, Dynamik und Modellierung, Universit\"at Stuttgart, Pfaffenwaldring 57, D-70569 Stuttgart, Germany}
\email{james.kennedy@mathematik.uni-stuttgart.de}

\address{Pavel Kurasov, Department of Mathematics, Stockholm University, SE-106 91 Stockholm, Sweden}
\email{kurasov@math.su.se}

\address{Delio Mugnolo, Lehrgebiet Analysis, Fakult\"at Mathematik und Informatik, Fern\-Universit\"at in Hagen, D-58084 Hagen, Germany}
\email{delio.mugnolo@fernuni-hagen.de}

\date{\today}

\thanks{The authors were partially supported by the Center for
  Interdisciplinary Research (ZiF) in Bielefeld in the framework of
  the cooperation group on ``Discrete and continuous models in the
  theory of networks''.  GB was also partially supported by NSF grant
  DMS-1410657. PK was also supported by the Swedish Research Council grant D0497301.}

\begin{document}

\begin{abstract}
  We derive a number of upper and lower bounds for the first nontrivial eigenvalue of
  a finite quantum graph in terms of the edge connectivity of the
  graph, i.e., the minimal number of edges which need to be removed to
  make the graph disconnected. On combinatorial graphs, one of the
  bounds is the well-known inequality of Fiedler, of which we give a
  new variational proof.   On quantum graphs, the corresponding bound 
  generalizes a recent result of Band and L\'evy.  All proofs are general 
  enough to yield corresponding estimates for the $p$-Laplacian and allow us to
  identify the minimizers.

  Based on the Betti number of the graph, we also derive upper and lower bounds
  on all eigenvalues which are ``asymptotically correct'', i.e. agree with the
  Weyl asymptotics for the eigenvalues of the quantum graph. In particular, the
  lower bounds improve the bounds of Friedlander on any given graph for all but
  finitely many eigenvalues, while the upper bounds improve recent results of
  Ariturk. Our estimates are also used to derive bounds on the eigenvalues of
  the normalized Laplacian matrix that improve known bounds of spectral graph
  theory.
\end{abstract}

\maketitle

\section{Introduction} 

The edge connectivity $\eta$ of a (combinatorial) graph $\mG$ is defined as the smallest number such that removal of at least $\eta$ edges is necessary in order to disconnect the graph; or equivalently, by Menger's theorem~\cite[Thm.~III.5]{Bollobas}, as the largest number of edge-independent paths that connect any two vertices of $\mG$. If we think of heat diffusion in a combinatorial graph, then the higher $\eta$, the more ways heat can spread around in the graph, regardless of the initial profile. This suggests that diffusion processes in graphs with larger edge connectivity can be expected to be faster spreading; this intuition was substantiated in a seminal article of Miroslav  Fiedler~\cite{Fie_cmj73}, where the \textit{spectral gap of a combinatorial graph} --- the lowest nonzero eigenvalue of the discrete Laplacian --- was shown to admit upper and lower estimates in terms of the edge connectivity of the same graph. More precisely, the main results in~\cite[\S~4]{Fie_cmj73} can be summarized as follows.

\begin{theorem}\label{thm:fiedler-2}
Let $\mG$ be a connected finite graph on $V$ vertices with edge connectivity
$\eta$. Then the first nontrivial eigenvalue $\gamma_1(\mG)$ of the discrete Laplacian $\mathcal L$ satisfies
\begin{equation}\label{eq:discconnectn}
2\eta\left[1-\cos \left(\frac{\pi}{V}\right) \right]
\le \gamma_1(\mG)
\le \eta+1.
\end{equation}
\end{theorem}
Fiedler's lower estimate~\eqref{eq:discconnectn} is sharp: for $\eta=1$ and any given $V$, the path graph on $V$ vertices $\{\mv_1,\ldots,\mv_V\}$has spectral gap $\gamma_1=2\left[1-\cos\left(\frac{\pi}{V}\right)\right]$ with associated eigenvector 
\begin{displaymath}
f_1(\mv_j)=\cos\left(\frac{\pi}{V}\left(j-\tfrac{1}{2}\right)  \right),\qquad j=1,\ldots,V,
\end{displaymath}
cf.~\cite[Lemma~2.4.5]{Spi09}; by allowing for graphs with multiple edges and replacing the edges of a path by $\eta$ parallel edges, one can easily prove that Fiedler's lower estimate is sharp for any $\eta$.  The upper estimate is also sharp: in fact, equality holds for complete graphs, the tighter inequality $\gamma_1(\mG)\le \eta{}$ holding for any other graph.

\medskip
By means of this and further results, Fiedler systematically investigated the spectral gap of a combinatorial graph, and in particular its dependence on the graph's connectivity. Thus, he paved the way for the birth of combinatorial spectral geometry, and he did so mostly by linear algebraic methods. Other estimates on $\gamma_1(\mG)$ can be obtained exploiting different techniques, as observed by Fiedler himself: indeed, it is well known that $\mathcal L=\mathcal I\mathcal I^T$, where $\mathcal I$ is the signed incidence matrix of an arbitrary orientation of the graph $\mG$ (a discrete counterpart of the divergence operator), and therefore
\begin{equation}
\label{eq:discrete-rayleigh-quotient}
	\gamma_1 (\mG) = \inf \left\{ \frac{\sum\limits_{\me \in \mE } |\mathcal I^T f(\me)|^2 }{ \sum\limits_{\mv \in \mV} 	|f(\mv)|^2} :\, 0 \neq f \in \R^\mV,\, \sum_{\mv \in \mV} f(\mv) = 0 \right\}.
\end{equation}
For instance, this formula immediately shows that deleting an edge decreases $\gamma_1$.

It is one of the main points of this paper to demonstrate that edge connectivity is in fact a natural quantity to study through its influence on the Rayleigh quotient; and that such methods extend to other types of graphs and operators thereon.  As a warm-up, we are going to provide in Section 2 an alternative, variational proof of Theorem~\ref{thm:fiedler-2}.  Our method allows us to characterize completely the case of equality \emph{and} to extend Fiedler's result with only minor changes to so-called $p$-Laplacians on graphs, a class of nonlinear operators that is gaining in prominence~\cite{hein2015mini}.  We will also obtain a Fiedler-like estimate on the lowest nonzero eigenvalue of the so-called normalized Laplacian of $\mG$, cf.\ Section~\ref{rem:otherthanfiedler}.

\medskip 

However, the main emphasis of this article is on differential Laplace operators on metric graphs -- that is, quantum graphs.  Let $\Graph$ be a graph where each edge $\me$ is identified with an interval $[0,\ell_{\me}]$ of the real line. This gives us a local variable $x_{\me}$ on the edge $e$ which can be interpreted geometrically as the distance from the initial vertex. Which of the two end vertices is to be considered initial is chosen arbitrarily; the analysis is independent of this choice.  

We are interested in the eigenproblem of the Laplacian
\begin{equation}
  -\frac{\partial^2}{\partial x^2} u_{\me}(x) = \lambda u_{\me}(x),
  \label{eq:metric_laplacian}
\end{equation}
where the functions $u$ are assumed to belong to the Sobolev space $H^{2}(\me)$
on each edge $\me$.  We will impose \emph{natural}\footnote{Also called Neumann, Neumann-Kirchhoff or standard conditions.} conditions at the vertices of the graph: we require that
$u$ be continuous on the vertices, $u_{\me_{1}}(\mv)=u_{\me_{2}}(\mv)$
for each vertex $\mv$ and any two edges, $\me_{1},\me_{2}$ incident to
$\mv$, and that the current be conserved,
\begin{equation}
  \label{eq:current_cons-2}
  \sum_{\me\sim \mv} \frac{\partial}{\partial x} u_{\me}(\mv) = 0
  \qquad \mbox{ for all vertices } \mv,
\end{equation}
where the summation is over all edges incident to a given vertex and the derivative is covariant into the edge (i.e.\ if $\mv$ is the final vertex for the edge $\me$, the whole term gets a minus sign). Further information can be found in the review \cite{GnuSmi_ap06}, the textbooks \cite{BerKuc_graphs,Mugnolo_book} or a recent elementary introduction \cite{Ber_prep16}.

Unless specified otherwise, the graphs we study are connected and compact; more precisely, we assume them to have finitely many edges, all of which have finite length. In this case, the spectrum of the Laplacian on the graph is discrete and the eigenvalues can be ordered by magnitude.  Under the conditions specified above, $0$ is always the lowest (and simple!) eigenvalue, hence we number our eigenvalues as
\begin{equation}
  \label{eq:eig_from0}
  0 = \eigO_0(\Graph) < \eigO_1(\Graph) \leq \eigO_2(\Graph) \leq \ldots
\end{equation}
It will be more natural to use a different numbering when we start changing our vertex conditions, see \eqref{eq:eig_from1} below.

In a recent article~\cite{K2M2_ahp16} three of the current authors (JBK, PK and DM) and Gabriela Malenová investigated whether certain combinatorial or metric quantities associated with a quantum graph, alone or combined, are sufficient to yield upper and/or lower estimates on \textit{the spectral gap of a quantum graph}, i.e., on the smallest nonzero eigenvalue $\eigO_1(\Graph)$ with natural vertex conditions, which can be characterized as
\begin{equation}\label{eq:rq}
\begin{split}
	\eigO_1(\Graph) &= \inf\left\{ \frac{\int_\Graph
            |u'(x)|^2 \dx}{\int_\Graph |u(x)|^2 \dx} :
	0\ne u \in H^1(\Graph),\ {\int_\Graph u(x) \dx}=0 \right\}\\
&=\inf\left\{ \mu>0:\int_\Graph \psi'(x)u'(x) \dx=\mu\int_\Graph \psi(x)u(x) \dx\hbox{ for some $\psi\in H^1(\Graph)$ and all }u\in H^1(\Graph)\right\}
	\end{split} 
\end{equation}
where $ H^1 (\Graph) $ denotes the space of functions belonging to the Sobolev space $ H^1 $ on every edge and continuous on the whole graph $\Graph$. Note that the second characterization in \eqref{eq:rq} simply corresponds to finding the smallest value $\mu>0$ for which there is a solution (i.e.~eigenfunction $\psi$) of the weak form of the eigenvalue equation \eqref{eq:metric_laplacian} with the given vertex conditions. The quantities considered in~\cite{K2M2_ahp16} were\footnote{To save space and avoid unnecessary repetition, we refer to \cite{K2M2_ahp16} for both general background information on the problems we are considering and the relevant graph-theoretic terminology.}
\begin{itemize}
\item the number $V$ of the graph's vertices (or usually just the number of \emph{essential} vertices, i.e., of all vertices with degree $\ne 2$)
\item the number $E$ of the graph's edges
\item the \textit{total length} $L$ of the quantum graph, i.e., the sum of all edges' lengths 
\item the \textit{diameter} $\diam{\Graph}$ of the quantum graph, i.e., the
  largest possible distance between any two points in the quantum
  graph (vertices or edges' internal points alike)
\end{itemize}
In particular, let us mention the upper and lower estimates
\begin{equation*}
  \frac{\pi^2}{L^2} \leq \eigO_1(\Graph) \leq \frac{\pi^2 E^2}{L^2}
\end{equation*}
obtained in \cite{Nic_bsm87} and \cite{K2M2_ahp16}, respectively.  The upper estimate is sharp in the case of both ``pumpkin'' and ``flower graphs'' --- somehow the quantum graph analog of complete graphs, as they are defined by the condition that any two edges are adjacent --- whereas the lower estimate becomes an equality in the case of path graphs. Further recently obtained eigenvalue estimates of a similar nature for quantum graph Laplacians can be found in \cite{BanLev_prep16,DelRos_amp16,Roh16_preprint}.

Whenever the graph has higher connectivity, and therefore no ``bottleneck'', one may hope for better lower estimates. In this article we are going to devote our attention to estimating $\eigO_1(\Graph)$, and more generally $\eigO_k(\Graph)$, by two further quantities related to ``connectedness'' of the graph, namely
\begin{itemize}
\item the \textit{edge connectivity} $\eta$ of the graph, and
\item the \textit{Betti number} $\beta$ of the graph, i.e., $E-V+1$, which is also the number of independent cycles in the graph.
\end{itemize} 
Deriving lower estimates on $\eigO_1(\Graph)$ based on $\eta$ will be the topic of Section~\ref{sec:qg-ec}: there, we prove an analog of Fiedler's estimate,
\begin{equation}\label{eq:eta-bound-intro}
	\eigO_1(\mathcal{G}) \geq \frac{\pi^2 \eta^2}{(L+(\eta-2)_+ \ell_{\max})^2},
\end{equation}
where $L>0$ is the total length of the graph, $\eta \in \mathbb{N}$ is the (algebraic or discrete) edge connectivity (see Definition~\ref{def:dec}) and $\ell_{\max}$ is the length of the longest edge of $\mathcal{G}$; see Theorem~\ref{thm:connectn}. This generalizes a recent theorem of Ram Band and Guillaume L\'evy \cite[Thm.~2.1(2)]{BanLev_prep16}, which dealt with the case $\eta=2$.
The correction term in the denominator involving $\ell_{\max}$ turns out to be necessary, as simple examples show; see Remark~\ref{rem:connectn1}.

We remark that the natural condition at a vertex of degree one reduces to the Neumann condition $u'(\mv) = 0$.  When we consider upper and lower estimates for high eigenvalues in Section~\ref{sec:qg_betti}, it becomes necessary also to allow Dirichlet conditions $u(\mv) = 0$ on vertices of degree one.  For eigenvalue problems with some Dirichlet conditions the more natural eigenvalue numbering is
\begin{equation}
  \label{eq:eig_from1}
  0 < \eigI_1(\Graph) < \eigI_2(\Graph) \leq \eigI_3(\Graph) \leq \ldots.
\end{equation}
For uniformity, we will use this numbering throughout Section~\ref{sec:qg_betti}, whether the graph has any Dirichlet vertices or not.  Using a symmetrization-based estimate (essentially derived in the proof of Theorem~\ref{thm:connectn}), an interlacing lemma from \cite{BanBerWey_jmp15} and other surgery principles (see, for example, \cite{BerKuc_incol12}, \cite[Section 3.1.6]{BerKuc_graphs} or \cite[\S~2]{KurMalNab_jpa13,K2M2_ahp16}), we provide concise proofs for several existing spectral gap estimates and discover new (or improved) ones.  In particular, for a tree with Dirichlet conditions imposed on \emph{all} vertices of degree one, we prove that
\begin{equation}
  \label{eq:diameter_estimate}
  \eigI_1(\Graph) \geq \frac{\pi^2}{D^2},
\end{equation}
where $D$ is the diameter of the graph.  Most notably, for \emph{all} eigenvalues
$\eigI_k(\Graph)$, we show that
\begin{equation}
  \label{eq:low_bound_all_eig_intro}
  \eigI_k(\Graph) \geq 
  \begin{cases}
    \left(k - \frac{|N|+\beta}2 \right)^2 \dfrac{\pi^2}{L^2}  
    & \mbox{if } k \geq |N| + \beta\\[10pt]
    \dfrac{k^2 \pi^2}{4 L^2} & \mbox{otherwise},
  \end{cases}
\end{equation}
which is an improvement of a lower bound by Friedlander \cite{Fri_aif05}, and 
\begin{equation}
  \label{eq:upp_bound_all_eig_intro}
  \eigI_k (\Graph)
  \leq \left(k - 2 + \beta + |D| + \tfrac{|N|+\beta}2\right)^2 \frac{\pi^2}{L^2},
\end{equation}
which is an improvement of a recent result of Ariturk \cite{Ari_prep16}.  Here we have denoted by $|D|$ and $|N|$ the number of vertices of degree one with Dirichlet and Neumann conditions, respectively; see Theorems~\ref{thm:low_bound_all_eig} and \ref{thm:upp_bound_all_eig} for further details and subsequent discussion.

Since we use variational techniques, we do not generally require linearity of the considered operators; effectively, only homogeneity of the Rayleigh quotients is used: this leads us to formulate most of our results for $p$-Laplace operators as well; in particular, this is true of our generalization of Fiedler's Theorem~\ref{thm:fiedler-2}, namely Theorem~\ref{thm:fiedler-p}, as well as  Theorem~\ref{thm:connectn}.  Let us recall their basic definition: 
For functions defined on the vertices $\mv\in \mV$ of a combinatorial graph, or on the edges $\me\in \mE$ of a metric graph, respectively, the $p$-Laplacian is defined as 
\begin{equation}\label{eq:p-defin-explicit-discr}
\Lfun_p f(\mv):=\mathcal I(|\mathcal I^Tf|^{p-2}\mathcal I^Tf)(\mv)=\sum_{\mw\sim \mv} \Big(|f(\mv)-f(\mw)|^{p-2}\big(f(\mv)-f(\mw)\big)\Big)\ ,\qquad \mv\in \mV\ 
\end{equation}
respectively as
\begin{equation}\label{eq:p-defin-explicit-contin}
\Delta_p u(x):=\left(|u'|^{p-2}u' \right)'(x),\qquad x\in \me,\ \me\in \mE
\end{equation}
(with suitable vertex conditions that will be specified later). The latter operators have not been extensively studied so far on graphs: we are only aware of~\cite{DelRos_amp16} and~\cite[\S~6.7]{Mugnolo_book}. Discrete $p$-Laplacians have a much longer history that goes back to~\cite{NakYam76}, see~\cite{BuhHei09,Mug_na13} for references. 
While the theory of $p$-Laplacians on domains and manifolds is very rich, and on graphs and quantum graphs would appear to have similar potential, all that we shall need is (nonlinear) variational eigenvalue characterizations entirely analogous to their linear counterparts: see \eqref{eq:discrete-rayleigh-quotient-p} for the combinatorial and \eqref{eq:cont-rayleigh-quotient-p} for the metric case, respectively. More theoretical background can be found, e.g., in~\cite[\S~3.2]{LanEdm11}.

\section{Combinatorial graphs: Estimates based on the edge connectivity} 
\label{sec:comb-graph}

We will always make the following assumption on our combinatorial graphs, which we will generically denote by $\mG$. We denote by $\mV$ and $\mE$ the vertex and edge set of $\mG$, respectively.
\begin{assumption}\label{assump:discrete}
The combinatorial graph $\mG$ is connected and consists of a finite number of vertices and edges. Multiple edges between given pairs of vertices are allowed, but loops are not.
\end{assumption}

Allowing for multiple edges turns out to be natural in view of the relevant role played by the edges in the Rayleigh quotient of the discrete Laplacian $\mathcal L$. Forbidding loops is a standard assumption in spectral graph theory, since the incidence value of a loop has no natural definition.

Most of our results are based on two general methods: symmetrization techniques (originally adapted to quantum graphs in~\cite{Fri_aif05}), and general graph surgery.  
  
The Laplacian of a combinatorial graph (possibly with multiple edges) satisfying Assumption~\ref{assump:discrete} is a square matrix $\mathcal L$ of size $V$. Its $\mv-\mw$-entry $\mathcal L_{\mv\mw}$ is defined as follows:
  \begin{displaymath}
  \mathcal L_{\mv\mw}:=\begin{cases}
  -\#\{\me\in \mE: \me \hbox{ is incident to both }\mv,\mw\}\qquad &\hbox{if }\mv\ne \mw\\
  \#\{\me \in \mE: \me \hbox{ is incident to }\mv\}&\hbox{if }\mv= \mw\ .\\
  \end{cases}
  \end{displaymath}
It is known that 
\begin{equation}\label{eq:laplincid}
\mathcal L=\mathcal I\mathcal I^T\ ,
\end{equation}
 where $\mathcal I$ is the signed incidence matrix of an arbitrary orientation of $\mG$: it maps $\mathbb R^\mE$ to $\mathbb R^\mV$ and is defined by
\begin{displaymath}
\mathcal I^T f(\me):=f(\me_{init})-f(\me_{term}),\qquad \me\in \mE\ ,
\end{displaymath}
where $\mV,\mE$ are the sets of vertices and edges of $\mG$,
respectively, and for an (arbitrary) orientation of $\mG$ we denote by
$\me_{init},\me_{term}$ the initial and terminal endpoint of $\me\in
\mE$, respectively. A straightforward calculation shows that the formula~\eqref{eq:laplincid}, which is classical in the case of graphs with no multiple edges, is still true in our more general setting.

It is immediate that 0 is a simple eigenvalue of $\mathcal L$ and the associated eigenspace is the space of constant functions. Therefore, we are interested in the lowest nonzero eigenvalue $\gamma_1(\mG)$ of $\mathcal L$: Fiedler's estimates in~\eqref{eq:discconnectn} clearly show that $\gamma_1(\mG)$ is a measure of the connectedness of the graph. We will now extend Fiedler's result to discrete $p$-Laplacians, by a method significantly different from Fiedler's.

The smallest nontrivial, i.e.~nonzero, eigenvalue $\gamma_1^{(p)}$ of the \emph{discrete $p$-Laplacian} on $\mG$, where $p \in (1,\infty)$, is given by
\begin{equation}
\label{eq:discrete-rayleigh-quotient-p}
	\gamma_1^{(p)}(\mG) := \inf \left\{ \frac{\sum\limits_{\me \in \mE } |f(\me_{init})-f(\me_{term})|^p }
	{ \inf\limits_{\gamma\in \R}\sum\limits_{\mv \in \mV}|f(\mv)-\gamma|^p} :\, const \neq f \in \R^\mV \right\}.
\end{equation}
We refer to~\cite{BuhHei09,Mug_na13} and the references therein for 
more information, including motivations to study the $p$-Laplacian. It follows from
simple compactness arguments that there is a vector $\efd^{(p)}$ achieving
equality in \eqref{eq:discrete-rayleigh-quotient-p}.

\begin{remark}
In nonlinear operator theory there are several competing approaches to spectral analysis, and even the mere definition of eigenvalue is in general not unequivocal. Some justification of the choice of~\eqref{eq:discrete-rayleigh-quotient-p} as definition of the lowest nonzero eigenvalue of the discrete $p$-Laplacian $\mathcal L_p$ is therefore in order. To begin with, observe that 
 $\gamma_1^{(p)}(\mG)$ becomes equal to $\gamma_1(\mG)$   when $p=2$. Indeed,
  since $\inf_{\gamma\in\R} \|f-\gamma\|_2^2$ is achieved for $\gamma$
  such that $(f-\gamma) \perp 1$, and since the numerator is invariant
  under the change $f \mapsto f - \gamma$, we only need to consider
  $f$ orthogonal to 1.  Then~\eqref{eq:discrete-rayleigh-quotient-p} reduces to~\eqref{eq:discrete-rayleigh-quotient}.
Furthermore, it is known that with the variational definition~\eqref{eq:discrete-rayleigh-quotient-p} $\gamma_1^{(p)}(\mG)$ is actually an eigenvalue, cf.~\cite[Thm.~3.1]{BuhHei09}: in other words, $f$ realizes the $\inf$ in~\eqref{eq:discrete-rayleigh-quotient-p} if and only if
  \begin{equation}\label{eq:eigenvdiscrclass}
  \mathcal L_p f=\gamma^{(p)}_1(\mG) |f|^{p-2}f\ .
  \end{equation}
  \end{remark}

\begin{figure}
  \centering
  \includegraphics{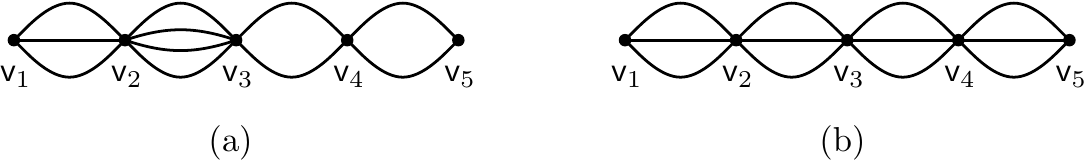}
  \caption{A non-regular pumpkin chain (a) and a 3-regular pumpkin chain (b).}
  \label{fig:pumpkins}
\end{figure}

\begin{theorem}\label{thm:fiedler-p}
Let $\mG$ be a combinatorial graph with edge connectivity $\eta$. Then for any $p\in (1,\infty)$ the eigenvalue $\gamma_1^{(p)}(\mG)$ of the discrete $p$-Laplacian on $\mG$, given by \eqref{eq:discrete-rayleigh-quotient-p}, satisfies
\begin{equation}
  \label{eq:discrete_p_estimate}
  \gamma_1^{(p)}(\mG)\ge \gamma_1^{(p)}(\widetilde \mG) = \eta\gamma_1^{(p)}(\mP_V)\ ,
\end{equation}
where $\gamma_1^{(p)}(\mP_V)$ denotes the corresponding eigenvalue of
the discrete $p$-Laplacian on a path graph $\mP_V$ on the same number
$V$ of vertices as $\mG$ and  $\widetilde \mG$ is an \emph{$\eta$-regular
pumpkin chain} with $V$ vertices (see below). Equality holds if and only if $\mG$
is an $\eta$-regular pumpkin chain. 
\end{theorem}

Fiedler's original proof of Theorem~\ref{thm:fiedler-2}, the special case of the
above theorem for $p=2$, is of purely linear algebraic nature and is based on
his own notion of \textit{measure of irreducibility}. Our proof is based on
standard variational techniques applied to \textit{pumpkin chains}. Pumpkins
(sometimes also called \emph{mandarins} or \emph{dipole graphs} in the
literature) are connected graphs on two vertices and pumpkin chains are the
graphs obtained from path graphs by possibly replacing existing edges by
multiple ones. We call a pumpkin chain $k$-\emph{regular} if any two adjacent
vertices are connected by the same number $k$ of multiple edges, see
Fig.~\ref{fig:pumpkins}. It is immediate that an $\eta$-regular pumpkin chain
has edge connectivity $\eta$ and eigenvalues equal to $\eta$ times those of a
path graph on the same number of vertices. In particular, its first eigenvalue
$\gamma_1$ is given by the right hand side of \eqref{eq:discconnectn}, so
Fiedler's estimate is sharp and $\eta$-regular pumpkin chains are minimizers of
$\gamma_1$ for given $V$ and $\eta$.

Our own proof of Theorem~\ref{thm:fiedler-p} is based on the principle that one can decrease the first eigenvalue of a combinatorial graph by replacing an edge with a suitably chosen path between the same two vertices, depending on the corresponding eigenvector. To this end, we need the variational characterization \eqref{eq:discrete-rayleigh-quotient-p} of $\gamma_1^{(p)}$.  We will follow the usual practice of calling the quotient in \eqref{eq:discrete-rayleigh-quotient-p} the \emph{Rayleigh quotient} and shall denote it by $\RQ\mG{f}$. In general, the minimizer may not be unique, however its precise choice is irrelevant to us; we shall choose one minimizer and denote it by $\efd$.

\begin{proof}[Proof of Theorem~\ref{thm:fiedler-p}]
  To prove the inequality between eigenvalues we reconnect the vertices of the
  graph $\mG$ in a way that reduces the Rayleigh quotient at every step.  Note
  that we will leave the set of vertices intact, thus the same function can be
  used as a test function for all intermediate graphs.
  
  Starting with a given minimizer $\efd$ on $\mG$, as described above, we order the vertices of $\mG$
  so that 
  \begin{equation}
    \label{eq:vertex_ordering}
    \efd(\mv_1) \leq \efd(\mv_2) \leq \ldots \leq \efd(\mv_V).  
  \end{equation}
  Suppose there is an edge whose endpoints are $(\mv_i,\mv_j)$, for some $i <
  j$, $j-i \geq 2$ (a ``long edge'').  We form a new graph $\mG'$ by replacing
  this edge by the sequence of edges with endpoints $(\mv_i, \mv_{i+1})$,
  $(\mv_{i+1}, \mv_{i+2})$, \ldots, $(\mv_{j-1}, \mv_j)$.  Here it is essential
  that we allow for pairs of vertices to be connected by multiple edges: we
  create the edges $(\mv_i, \mv_{i+1})$, $(\mv_{i+1}, \mv_{i+2})$, \ldots,
  $(\mv_{j-1}, \mv_j)$ \emph{in addition} to the already existing edges between
  these vertices, see Figure~\ref{fig:path_replacement}.

  \begin{figure}
    \centering
    \includegraphics[scale=1]{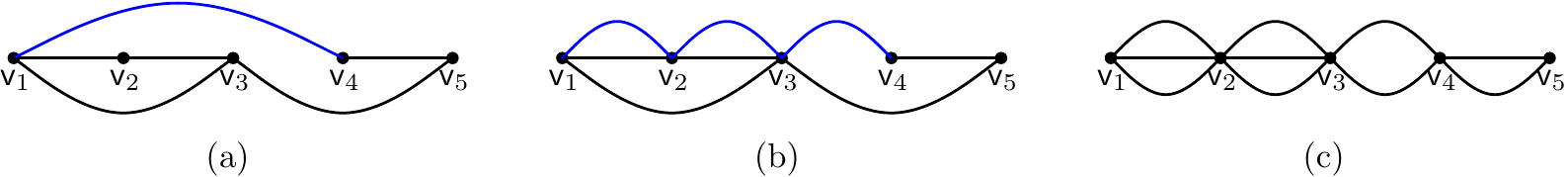}
    \caption{Replacing an edge by a path: (a) the original graph with
      the edge $\mv_1 \mv_4$ to be replaced shown in blue, (b) the graph
      obtained after replacing $(\mv_1,\mv_4)$ by three edges with endpoints
      $(\mv_1, \mv_2)$, $(\mv_2, \mv_3)$ and $(\mv_3 ,\mv_4)$ respectively (shown in blue), (c)
      the final graph $\widehat{\mG}$ after all replacements prescribed by
      the proof of Theorem~\ref{thm:fiedler-p}.}
    \label{fig:path_replacement}
  \end{figure}
  
  We claim that this operation decreases the Rayleigh quotient of the
  function $\efd$,
  \begin{displaymath}
    \RQ{\mG'}\efd \leq \RQ{\mG}\efd.
  \end{displaymath}
  Indeed, this is equivalent to the assertion
  \begin{equation}
    \label{eq:g-reduction}
    \sum_{\ell = i}^{j-1} |\efd(\mv_{\ell+1}) - \efd(\mv_{\ell})|^p 
    \leq |\efd(\mv_j)-\efd(\mv_i)|^p.
  \end{equation}
  Now by using the inequality $\|x\|_p \leq \|x\|_1$ between the $p$- and
  $1$-norms of a finite-dimensional vector (which is strict as long as $p>1$ and
  the vector has more than one nonzero component), applied to
  \begin{equation}\label{eq:x}
    x:=\big(\efd(\mv_{i+1})-\efd(\mv_i),\ldots,\efd(\mv_j)-\efd(\mv_{j-1})\big) \in \R^{j-i},
  \end{equation}
  we obtain
  \begin{align*}
    \left[\sum_{\ell=i}^{j-1} |\efd(\mv_{\ell+1})-\efd(\mv_\ell)|^p
    \right]^{1/p} 
    &\leq \sum_{\ell=i}^{j-1}|\efd(\mv_{\ell+1})-\efd(\mv_\ell)| \\
    &= \sum_{\ell=i}^{j-1} \efd(\mv_{\ell+1})-\efd(\mv_\ell) = \efd(\mv_j)-\efd(\mv_i),
  \end{align*}
  where the first equality here follows from the ordering of the vertices,
  equation~\eqref{eq:vertex_ordering}.  Taking the $p$-th power of both sides
  yields \eqref{eq:g-reduction}.
  
  Repeating this operation for every ``long edge'', we obtain a pumpkin chain
  $\widehat\mG$, which is in general non-regular.  Since the graph operation we
  performed cannot decrease the edge connectivity, $\widehat\mG$ has edge connectivity at least
  $\eta$, that is, each pumpkin in $\mG_1$ has at least $\eta$ parallel edges.  We
  now remove excess edges, producing an $\eta$-regular pumpkin chain $\widetilde
  \mG$.  It is immediate to see that removing an edge between $v_i$ and $v_{i+1}$ also
  decreases the Rayleigh quotient (strictly decreases if $\efd(v_i) \neq \efd(v_{i+1})$).

  We therefore have
 \begin{displaymath}
    \gamma_1^{(p)} (\widetilde \mG) \leq \RQ{\widetilde \mG}\efd
    \leq \RQ{\widehat\mG}\efd
    \leq \RQ{\mG}\efd
    = \gamma_1^{(p)} (\mG).
  \end{displaymath}
  This proves the inequality. 
  
  To characterize the case of equality we note that it can only happen if the
  original function $f_1$ is also a minimizer of the final graph $\widetilde
  \mG$ and thus an eigenfunction.  However, it can be shown that any 
  first nontrivial eigenfunction $f_1$
  of the discrete $p$-Laplacian on a regular pumpkin chain
  $\widetilde \mG$ (equivalently, on a path graph) does not take on the same
  value on neighboring vertices (see e.g.~\cite[Lemma 2.4.5]{Spi09} for the
  special case of $p=2$).

  For such a test function $f_1$, every subdivision of a ``long edge'' strictly
  decreases the Rayleigh quotient; so does every removal of an edge.  We
  conclude that no operations can have been performed and therefore $\widetilde
  \mG = \mG$.
\end{proof}

\section{Quantum graphs: Estimates based on the edge connectivity} 
\label{sec:qg-ec}

Let us now turn to the case of quantum graphs. It is natural to ask whether we can obtain a result similar to
Theorem~\ref{thm:fiedler-2} for the smallest nontrivial eigenvalue $\eigO_1(\Graph)$ of the Laplacian with natural vertex conditions on the metric graph $\mathcal{G}$, see~\eqref{eq:rq}. 
We recall that these vertex conditions require the functions to be continuous at every vertex and the sum of normal derivatives to be zero. But since we are going to work with quadratic forms, it is enough to require that our test functions be continuous at every vertex.

More generally, as in the discrete case, we will in fact consider the smallest nontrivial eigenvalue of the $p$-Laplacian with natural conditions. More precisely, we consider the following weak eigenvalue problem: we call $\mu>0$ an eigenvalue, and $\psi \in W^{1,p}(\Graph)$ an eigenfunction, if
\begin{equation}
\label{eq:p-lapl-weak-form}
	\int_{\Graph} |\psi'(x)|^{p-2}\psi'(x) u'(x)\dx=\mu\int_{\Graph} |\psi(x)|^{p-2}\psi(x)u(x)\dx\qquad
	\text{for all } u \in W^{1,p}(\Graph),
\end{equation}
for fixed $p \in (1,\infty)$. Here the Sobolev space $W^{1,p}(\Graph) \subset C(\Graph)$ is defined in the usual way, as the vector space of those edgewise $W^{1,p}$-functions that satisfy continuity conditions in the vertices. The equation \eqref{eq:p-lapl-weak-form} is the weak form of an eigenvalue problem for the $p$-Laplace operator $-\Delta_p$ given by $-\Delta_p u = -(|u'|^{p-2}u')'$ on the same metric graph $\Graph$, again with continuity conditions in the vertices along with Kirchhoff-type conditions
\begin{equation}
  \label{eq:current_cons}
  \sum_{\me\sim \mv} \left|\frac{\partial}{\partial x} u_{\me}(\mv)\right|^{p-2}\frac{\partial}{\partial x} u_{\me}(\mv) = 0
  \qquad \mbox{ for all vertices } \mv,
\end{equation}
where the summation is again taken over all edges incident to $\mv$ and the derivative is covariant into the edge. (We stress the similarity to~\eqref{eq:eigenvdiscrclass}.) If $p=2$, we obviously recover the usual Laplacian. Here and below, we will consider the quantity
\begin{equation}\label{eq:cont-rayleigh-quotient-p}
	\eigO_1^{(p)}(\mathcal{G}):=\inf \{\mu>0: \mu \text{ is an eigenvalue in the sense of \eqref{eq:p-lapl-weak-form}}\}.
\end{equation}
When $p=2$, we recover the smallest nontrivial Laplacian eigenvalue, i.e., $\eigO_1^{(2)}(\mathcal{G})=\eigO_1(\mathcal{G})$ defined by \eqref{eq:rq}.\footnote{We also expect for general $p\in (1,\infty)$ that the infimum in \eqref{eq:p-lapl-weak-form} will be a minimum, i.e., $\eigO_1^{(p)}(\Graph)$ will be an eigenvalue. In the case of intervals, this is known, and indeed, the corresponding eigenvalue has a variational characterization in terms of the \emph{Krasnoselskii genus}, cf.~\cite[Chapt.~3]{LanEdm11} and~\cite{BinRyn08} for more information and a description of various variational quantities which, in competing senses, are nonlinear generalizations of $\mu_1^{(2)}$. On graphs these variational characterizations seem to be unknown, although we very much expect them to hold.}

We shall also have occasion to consider the first eigenvalue of the $p$-Laplacian with the Dirichlet (zero) condition on a subset $\mV_D$ of the set of vertices of $\mathcal{G}$. In this case, we shall write $W^{1,p}_0 (\mathcal{G}) := \{u \in W^{1,p}(\mathcal{G}): u(\mv)=0 \text{ for all } \mv \in \mV_D\}$ and
\begin{equation}
\label{eq:dirichlet-rayleigh-quotient-p}
	\eigI_1^{(p)}(\mathcal{G}) := 
	\inf \left\{  \frac{\int_{\Graph} |u'(x)|^p\dx}{\int_{\Graph} |u(x)|^p\dx} :
	u \in W^{1,p}_0(\Graph) \right\}\ ,\qquad p\in (1,\infty)
\end{equation}
(note that this is now always an eigenvalue, and indeed the very smallest). Of course, both $W^{1,p}_0 (\mathcal{G})$ and $\eigI_1^{(p)} (\mathcal{G})$ depend on the set $\mV_D$, but in practice it will always be clear which set this is.

We shall keep the following standing assumption on $\mathcal{G}$ throughout the paper, which mirrors Assumption~\ref{assump:discrete} for combinatorial graphs.

\begin{assumption}
  \label{assumption:finite}
 The metric graph $\mathcal{G}$ is connected, compact and finite,
  i.e., it is formed by a  finite
  number $E$ of edges of finite length connected together at a finite number $V$ of vertices.
Moreover, we assume that $\mathcal{G}$ does not have any vertices of degree two. \end{assumption}

Note that we allow loops (except when we explicitly assume the contrary) and multiple edges.

In order to explain the assumption about vertices of degree two, let us recall that introducing such a vertex for $ p=2 $, natural vertex conditions imply that
both the function and its first derivative are continuous across the vertex. Hence the vertex may be removed and the two edges may be substituted by
one longer edge so that the eigenvalues of the quantum graph are preserved as well as its eigenfunctions up to a canonical identification. This is a well-known
phenomenon in the theory of quantum graphs when $p=2$. For $ p \neq 2$ exactly the same phenomenon can be observed, since any continuous and piecewise-$W^{1,p}$ function belongs to $W^{1,p}$.

We now wish to introduce a notion of edge connectivity for a metric graph $\mathcal{G}$. To cut an edge means to divide it into two smaller edges
with the sum of lengths equal to the length of the original edge, keeping the vertices of the original graph and introducing two new degree one vertices.  
Then it is natural to use the following definition:

\begin{definition}\label{def:connectivity}
  Let $\Graph$ be a metric graph. The \emph{metric edge connectivity} $\eta$
  of $\Graph$ is equal to the number of edge cuts needed to make the graph
  disconnected.
\end{definition}

It is clear that the metric edge connectivity is at most two, since one may
always cut any given edge twice.  In the recent paper \cite{BanLev_prep16}, Band and
L\'evy obtained the lower bound
\begin{equation}\label{eq:bl16}
	\eigO_1(\Graph) = \eigO_1^{(2)}(\Graph) \geq \frac{4\pi^2}{L^2}
\end{equation}
for all quantum graphs of length $L>0$ and edge connectivity two in this sense.
This result includes, in particular, Eulerian graphs, for which the same
bound was proved in \cite[Thm.~2]{KurNab_jst14}.

The proof of Band and L\'evy is based on a refinement of the symmetrization argument introduced by Friedlander \cite{Fri_aif05} used in his proof of the isoperimetric inequality $\eigO_1(\Graph) \geq \pi^2/L^2$. This latter inequality together with \eqref{eq:bl16} may be thought of as a natural counterpart to Fiedler's estimate if we regard the edge connectivity of a quantum graph as being always either one or two.

We shall also introduce \emph{discrete edge connectivity}, which we believe is better adapted to quantum graphs (despite its name!). In particular, with it we can obtain better bounds if the discrete edge connectivity is more than two.

We have already noticed that degree two vertices can be removed without really affecting the metric graph, leading to equivalent graphs. 
The only exception is the cycle graph which contains just one vertex,
which of course is impossible to eliminate. The discrete edge connectivity will not distinguish between equivalent metric graphs.

\begin{definition}\label{def:dec}
  Let $\Graph$ be a finite metric graph.
  If the graph possess vertices of degree two, then consider the equivalent metric graph obtained from $ \Graph $ by removing
  vertices of degree two (if possible).
   If the number of vertices in the new graph is greater than one, then  the \emph{discrete edge
    connectivity} $\eta$ of $\Graph$ is the minimal number of edges that
  need to be deleted in order to make the new graph disconnected.

 In the degenerate case of a graph with one vertex (a flower graph including the cycle graph as a special case) we let $\eta = 2$.
\end{definition}

The discrete edge connectivity is finer; it coincides with the metric one when it is equal to one or two.  Henceforth we will use only the discrete connectivity while keeping the same notation $\eta$.

Our main result is as follows.

\begin{theorem}\label{thm:connectn}
Let $\Graph$ be a connected quantum graph of total length $L$ and discrete edge connectivity $\eta$, and let $p \in (1,\infty)$. If $\eta \geq 2$, then
\begin{equation}\label{eq:connectp2}
	\eigO_1^{(p)}(\mathcal{G}) \geq (p-1) \left(\frac{2\pi_p}{L}\right)^p .
\end{equation}
Moreover, if $\eta \geq 3$, then
\begin{equation}\label{eq:connectn}
	\eigO_1^{(p)} (\Graph)\geq (p-1)\left(\frac{\eta\pi_p}{L+\ell_{\max}(\eta-2)}\right)^p,	
\end{equation}
where $\ell_{\max}$ is the length of the longest edge of $\mathcal{G}$. 
\end{theorem}

Here
\begin{displaymath}
\pi_p:=\frac{2\pi}{p\sin\frac{\pi}{p}}\ .
\end{displaymath}
It is known that $\pi_p$ is the smallest positive root of $\sin_p$, defined on $[0,\frac{\pi_p}{2}]$ as the inverse of 
\begin{displaymath}
s\mapsto \int_0^x (1-s^p)^{-\frac{1}{p}}\ ds
\end{displaymath}
and then suitably extended to $\R$ by reflection and periodicity.\footnote{The so-called $p$-trigonometric functions $\sin_p$ and $\cos_p:=\sin_p'$ (with $1/p+1/p'=1$) play in the theory of $p$-Laplacians a similar role to the trigonometric functions in the case $p=2$, but with some notable differences; for instance, while $\sin_p$ is an eigenfunction for the $p$-Laplacian with Dirichlet boundary conditions on $[-\frac{\pi_p}{2},\frac{\pi_p}{2}]$, its derivative $\cos_p:\R\to \R$ is continuous but not twice continuously differentiable for $p<2$, and not even continuously differentiable if $p>2$.} The values on the right hand sides of \eqref{eq:connectp2} and \eqref{eq:connectn} correspond to the first nontrivial eigenvalue of the Neumann Laplacian on an edge of length $L/2$ and $(L+\ell_{\max}(\eta-2))/\eta$, respectively. See, e.g., the monograph~\cite{LanEdm11} for more details.

\begin{corollary}\label{cor:connectncor}
Under the assumptions of Theorem~\ref{thm:connectn} we have in particular for $p=2$
\begin{equation}\label{eq:connect2n}
	\eigO_1 (\mathcal{G}) \geq \frac{\eta^2\pi^2 }{(L+\ell_{\max}(\eta-2)_+)^2}.
\end{equation}
\end{corollary}

\begin{remark}
\label{rem:connectn1}
(a) The estimate \eqref{eq:connectp2} is sharp, but \eqref{eq:connectn} is not, except asymptotically for certain sequences of graphs for which $\ell_{\max} \to 0$, 
provided $\eta$ is kept fixed. To give examples, we will need details from the proof of Theorem~\ref{thm:connectn} (including classes of graphs which will be introduced there). Hence we defer the proof of both statements until after this proof; see Remark~\ref{rem:equality}. There, we also characterize the case of equality in \eqref{eq:connectp2}. This was already treated in \cite{BanLev_prep16} when $p=2$.

(b) A correction term along the lines of $(\eta-2) \ell_{\max}$ is however necessary if $\eta \geq 3$. It may be viewed as consequence of using the discrete edge connectivity. Indeed, a lower bound of the form $\pi_p^p \eta^p/L^p$ is obviously impossible, even for $p=2$: take, for example, a very small pumpkin on $\eta$ slices and attach a single, very long loop to one of its vertices. 

(c) The bound in \eqref{eq:connectn} is an increasing function of $\eta \geq 2$ if (and only if) $L \geq 2\ell_{\max}$. This follows from the fact that the length of the comparison interval, i.e., $(L+(\eta-2)_+ \ell_{\max})/\eta$, is a decreasing function of $\eta \geq 2$ if and only if $L \geq 2\ell_{\max}$, together with the fact that the first eigenvalue of an interval is a decreasing function of the latter's length, for any $p \in (1,\infty)$. An inspection of the proof of Theorem~\ref{thm:connectn} shows that in fact
\begin{displaymath}
	\eigO_1^{(p)}(\mathcal{G}) \geq
	(p-1) \left(\frac{h\pi_p}{L+(h-2)\ell_{\max}}\right)^p
\end{displaymath}
for all $1 \leq h \leq \eta$ (just replace $\eta$ with $h$ throughout). Hence we may obtain a better bound than \eqref{eq:connectn} by taking the maximum over such $h$ if $L < 2\ell_{\max}$. Moreover, in the case $L < 2\ell_{\max}$, i.e., if a single edge ``dominates'' $\mathcal{G}$, then, by taking as a test function an appropriate $p$-sinusoidal curve on the longest edge extended by zero elsewhere on $\mathcal{G}$, we can obtain a trivial upper bound which shows $\eigO_1^{(p)}$ is essentially proportional to the longest edge length, meaning algebraic quantities such as edge connectivity are less relevant. For example, if $p=2$, this test function argument yields
\begin{displaymath}
	\eigO_1(\Graph) \leq \frac{4\pi^2}{(\ell_{\max})^2}
\end{displaymath}
to complement the lower bound of \eqref{eq:bl16}, namely
\begin{displaymath}
	\eigO_1(\Graph) \geq \frac{4\pi^2}{L^2} \geq \frac{\pi^2}{(\ell_{\max})^2}.
\end{displaymath}
\end{remark}

Let us now prepare for the proof of Theorem~\ref{thm:connectn}. We will need the following simple lemma, which provides an important technical link between the edge connectivity (in both discrete and continuous form) and the behavior of any $p$-Laplacian eigenfunction.

We will first need the following notation: denoting by $\mu>0$ any given eigenvalue of \eqref{eq:p-lapl-weak-form} and by $\psi \in W^{1,p} (\Graph) \subset C(\Graph)$ any eigenfunction corresponding to $\mu$, we set
\begin{equation}\label{eq:Gpm}
\begin{split}
	\mathcal{G}_+ &:= \{x \in \Graph : \psi (x) > 0\},\\
	\mathcal{G}_- &:= \{ x \in \Graph : \psi (x) < 0 \},\\
	\lvlsf(t) &:= \{ x \in \Graph : \psi (x) = t \}, \qquad t \in \R,
\end{split}
\end{equation}
noting that since $\mu>0$, both $\mathcal{G}_+$ and $\mathcal{G}_-$ must be nonempty. (Just take $u(x)=1$ in \eqref{eq:p-lapl-weak-form} to obtain the ``$p$-orthogonality condition'' $\int_{\Graph} |\psi(x)|^{p-2}\psi(x)\dx=0$, which forces $\psi$ to change sign in $\Graph$.) In general, $\mathcal{G}_\pm$ will not be connected. We also denote by
\begin{displaymath}
	\nu (t) := \# \lvlsf(t)
\end{displaymath}
the size of the level ``surface'' of $\psi$ at $t$, noting this is
finite whenever $t \neq 0$ as our graph is finite; it is possible that
$\psi$ is identically zero on some edge of $\Graph$, in which case
we write $\nu (0) = \infty$; observe that $0$ is the only constant
value that can be attained by an eigenfunction of a nonzero eigenvalue
on an open interval in $\Graph$. Finally, we denote by
\begin{displaymath}
  \excset := \psi (\mV) \equiv \{\psi (\mv) : \mv \in \mV \}
\end{displaymath}
the (finite) set of all values $\psi$ attains at the vertex set of $\Graph$ and by
\begin{displaymath}
	\excset^\ast := \excset \cup \{ \max_{x \in \Graph} \psi (x), \min_{x \in \Graph} \psi (x) \}
\end{displaymath}
the complete (and still finite) set of ``exceptional values'' of $\psi$.

\begin{lemma}
  \label{lemma:sizeofnu}
  Suppose with the above notation that the discrete edge connectivity $\eta \geq 2$. Then
  \begin{enumerate}
  \item[(i)] $\nu (t) \geq 2$ for all $t \in (\min_{x \in \Graph}
    \psi (x), \max_{x \in \Graph} \psi (x)) \setminus \excset$;
  \item[(ii)] $\nu (t) \geq \eta$ for all $t \in (\min \excset, \max \excset) \setminus \excset$.
  \end{enumerate}
\end{lemma}

\begin{proof}
  (i) If $t \in (\min_{x \in \Graph} \psi (x), \max_{x \in
    \Graph} \psi (x))$, then the graph $\Graph \setminus
  \lvlsf(t)$ must be disconnected, as it cannot contain any further
  path from the location of the minimum of $\psi$ to the location of the
  maximum. That is, removing $\lvlsf(t)$ cuts the graph into (at least) two
  pieces. Since $t \not \in \excset $, we have $\lvlsf(t) \cap \mV =
  \emptyset$. Our assumption on connectivity implies $\nu (t) = \# \lvlsf(t)
  \geq 2$.

  (ii) Similarly, for $t \in (\min \excset, \max \excset )$, the graph
  $\Graph \setminus \lvlsf(t)$ is disconnected, since there is no
  path left connecting the vertices of $\Graph$ at which $\min
  \excset$ and $\max \excset$ are attained. Since $t \not\in
  \excset$, the set of edges containing $\lvlsf(t)$ must have
  cardinality at least $\eta$, since its removal leads to a
  disconnected (discrete) graph.
\end{proof}

We emphasize that the restriction in (ii) to those values of $\psi$
between its minimal and maximal \emph{vertex} values is essential, as
is easy to see in the following example: if $\psi$ has a unique
maximum achieved in the interior of an edge, then for $t$ close to
this value we will have $\nu (t) = 2$ regardless of the discrete edge
connectivity.

\begin{proof}[Proof of Theorem~\ref{thm:connectn}]
  The proof is based on a sharper version of the symmetrization
  procedure of Friedlander \cite[Lemma~3]{Fri_aif05}, which was
  introduced on Eulerian graphs in the work of Kurasov and Naboko
  \cite{KurNab_jst14} and on graphs where $\eta=2$ in the work of Band and
  Levy \cite{BanLev_prep16}.  Here we sharpen this approach further
  for higher values of $\eta$.

We start by fixing an arbitrary eigenvalue $\mu>0$ and any eigenfunction $\psi$ associated with it; this also fixes the corresponding sets $\mathcal{G}_+$, $\mathcal{G}_-$, as well as the function $\nu(t)$ and so forth. Obviously, it suffices to prove the theorem for this $\mu>0$.

We first observe that we may assume without loss of generality that
\begin{equation}
\label{eq:max-vertex}
	\max \excset = \esssup \{ t \in \R : \nu (t) \geq \eta \},
\end{equation}
and that this maximum is attained at a single vertex, call it $\mv_{\max}$. Indeed, Lemma~\ref{lemma:sizeofnu}(ii) yields ``$\leq$'', and if we have strict inequality, then we may insert artificial vertices at all points $x \in \mathcal{G}$ where $\psi(x)$ is equal to this supremum and form a new graph, call it $\Graph'$, by identifying them. Then $\Graph'$ still has edge connectivity $\geq \eta$ by construction. Moreover, $\psi$ may be canonically identified with a function in $W^{1,p}(\Graph')$; moreover, since $W^{1,p}(\Graph')$ may be canonically identified with a subset of $W^{1,p}(\Graph)$, it follows that \eqref{eq:p-lapl-weak-form} continues to hold on $\Graph'$. Hence $(\mu,\psi)$ is, up to this identification, also an eigenpair of $\Graph'$ and we may consider it instead. To keep notation simple, we will write $\Graph$ in place of $\Graph'$.

If $\eta=2$, then $\psi$ attains its maximum over $\mathcal{G}$ at $\mv_{\max}$; if $\eta \geq 3$, this is not guaranteed.

We may also assume without loss of generality that $|\mathcal{G}_+| \leq L/2$ (where $|\,\cdot\,|$ denotes Lebesgue measure, i.e., length). We will denote by $W^{1,p}_0 (\mathcal{G}_+)$ those $W^{1,p}(\mathcal{G}_+)$-functions which are zero on the boundary of $\mathcal{G}_+$ in $\mathcal{G}$, which by standard properties of Sobolev functions may be characterized by
\begin{displaymath}
W^{1,p}_0 (\mathcal{G}_+) = \{u \in W^{1,p}(\mathcal{G}_+): \text{ there exists } \tilde u \in W^{1,p}(\mathcal{G})
	\text{ with } \tilde u \equiv 0 \text{ on } \mathcal{G}\setminus \mathcal{G}_+,\,u=\tilde u|_{\mathcal{G}_+} \}.
\end{displaymath}
In particular, $\psi \equiv \psi|_{\mathcal{G}_+} \in W^{1,p}_0 (\mathcal{G}_+)$. Moreover,
\begin{displaymath}
	\mu = \frac{\int_{\mathcal{G}_+} |\psi'|^p\,dx}{\int_{\mathcal{G}_+}|\psi|^p\,dx}\ ,
\end{displaymath}
as can be seen by taking $\psi_+ :=\psi\cdot \chi_{\mathcal{G}_+} \in W^{1,p}(\mathcal{G})$ as a test function in \eqref{eq:p-lapl-weak-form}.

We now associate with $\mathcal{G}_+$ a ``symmetrized'' half pumpkin with loops attached to one end -- a \emph{stower}, in the terminology of the recent paper~\cite{BanLev_prep16}. More precisely, we start by considering the set which we shall wish to symmetrize,
\begin{displaymath}
	\mathcal{G}_\excset:= \{ x \in \mathcal{G}_+ : \psi (x) \leq \max \excset \}.
\end{displaymath}
By definition of, and our assumptions on, $\excset$, the open set
$\mathcal{G}_+ \setminus \mathcal{G}_\excset$ has no vertices and thus
consists of some number $n \geq 0$ of individual edges, each having both
endpoints at $\mv_{\max}$, say $e_1,\ldots,e_n$, which represent the
set where $\psi > \max \excset$. By \eqref{eq:max-vertex}, $2n < \eta$; in particular, $n=0$ if $\eta=2$.

We now form a new graph $\mathcal{G}_+^\ast$ by
attaching $\eta$ equal edges $\tilde e_1,\ldots, \tilde e_\eta$ of
length $|\mathcal{G}_\excset|/\eta$ to each other at a common vertex
$\mv_0$ (i.e.~to form a star; for each $i$, we denote the other vertex of edge $\tilde e_i$ by $\tilde
\mv_i$). We identify each edge $\tilde e_i$ with the interval $[0,|\mathcal{G}_\excset|/\eta]$, 
whose right hand endpoint corresponds to the central vertex $\mv_0$. 
To $\mv_0$ we also attach $n$ loops $\hat
e_1,\ldots,\hat e_n$ emanating from $\mv_0$, having lengths
$|e_1|,\ldots,|e_n|$, respectively, see Fig.~\ref{fig:stower}(a).

\begin{figure}
  \centering
  \includegraphics{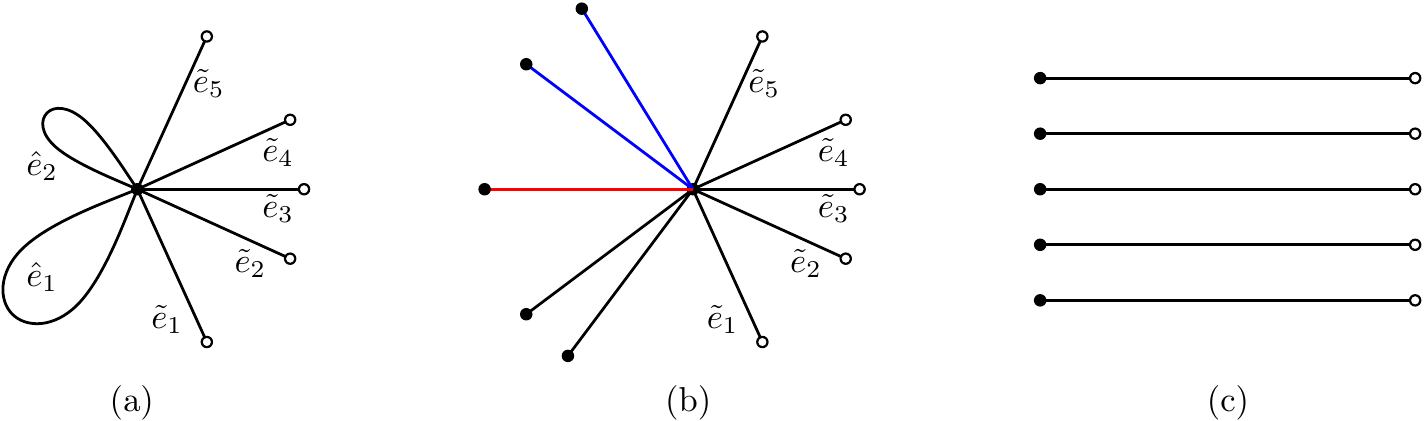}
  \caption{The symmetrized graph $\mathcal{G}_+^\ast$ and its modifications. Black and white dots represent vertices with Neumann and Dirichlet boundary conditions, respectively. (b) shows the graph $\widehat{\mathcal{G}}_+$, which has been formed from (a) by lengthening $\hat e_2$ (edges in blue), splitting the loops apart, and adding an edge (red) so that there are as many edges with a Neumann endpoint as Dirichlet ones.}
  \label{fig:stower}
\end{figure}

We now construct, for each $\phi \in W^{1,p}_0 (\mathcal{G}_+)$, a ``symmetrized'' function $\phi^\ast$ on $\mathcal{G}_+^\ast$ as follows: on each of the edges $\tilde e_i \simeq [0,|\mathcal{G}_\excset|/\eta]$ we set $\phi^\ast (0)=0$ and
\begin{equation}\label{eq:eta-symm}
	|\{x \in (0,|\mathcal{G}_\excset|/\eta): \phi^\ast (x) < t \}| = \frac{1}{\eta}\,|\{x \in \mathcal{G}_\excset : \phi (x) < t \}|
\end{equation}
for all $t \in \R$. We then set $\phi^\ast (\mv_0):=\phi(\mv_{\max})$ and extend $\phi^\ast$ to the loops $\hat e_i$ in the obvious way, by mapping the values of $\phi$ on each loop $e_i$ onto the corresponding $\hat e_i$, $i=1,\ldots,n$. By standard results on symmetrization, the resulting function $\phi^\ast$ is continuous across $\mathcal{G}_+^\ast$ and in $W^{1,p}$ on each edge, so we conclude $\phi^\ast \in W^{1,p}(\mathcal{G}_+^\ast)$.

Moreover, by construction and Cavalieri's principle,
\begin{displaymath}
	\|\phi^\ast\|_{L^p (\mathcal{G}_+^\ast)} = \|\phi\|_{L^p (\mathcal{G}_+)}\ ,
\end{displaymath}
and 
\begin{displaymath}
	\sum_{x \in \mathcal{G}(t)} \frac{1}{|\psi'(x)|} = \frac{\eta}{|(\psi^\ast)'(y_t)|}
\end{displaymath}
for almost all $t \in (0,\max \excset)$, where $y_t$ is any point in $\bigcup_{i=1}^\eta \tilde e_i$ such that $\psi^\ast (y_t)=t$.

Obviously,
\begin{displaymath}
	\|(\phi^\ast)'\|_{L^p\left(\bigcup_{i=1}^n \hat e_i \right)} = \|\phi'\|_{L^p \left( \bigcup_{i=1}^n e_i \right)}.
\end{displaymath}
We will now show for the eigenfunction $\psi$ and its symmetrization $\psi^\ast$ that
\begin{equation}\label{eq:p-symm-part}
	\|(\psi^\ast)'\|_{L^p\left(\bigcup_{i=1}^\eta \tilde e_i \right)} \leq \|\psi'\|_{L^p (\mathcal{G}_\excset)}.
\end{equation}
(In fact, \eqref{eq:p-symm-part} holds with the same argument for any $\phi \in W^{1,p}_0(\mathcal{G}_+)$, but this keeps the notation simpler and certain technical issues involving regularity do not arise.\footnote{It is known that any eigenfunction of the $p$-Laplacian on an interval is at least $C^{1}$, see~\cite[Chapters 2--3]{LanEdm11}. It follows immediately that the same is true for any graph eigenfunction on each edge.}) To see this, we use ideas similar to \cite[Lemma~3]{Fri_aif05} but make use of Lemma~\ref{lemma:sizeofnu}. By the coarea formula,
\begin{displaymath}
	\int_{\mathcal{G}_\excset} |\psi'|^p\,dx = \int_0^{\max \excset} \sum_{x \in \mathcal{G}(t)} |\psi'(x)|^{p-1}\,dt.
\end{displaymath}
By H\"older's inequality and Lemma~\ref{lemma:sizeofnu}(ii), for almost all $t \in (0,\max \excset)$, we have
\begin{displaymath}
	\nu(t)=\sum_{x \in \mathcal{G}(t)}1 = \sum_{x \in \mathcal{G}(t)} |\psi'(x)|^\frac{p-1}{p}\cdot
	\left(\frac{1}{|\psi'(x)|}\right)^\frac{p-1}{p} \leq \left(\sum_{x \in \mathcal{G}(t)}|\psi'(x)|^{p-1}\right)^\frac{1}{p}
	\left(\sum_{x \in \mathcal{G}(t)}\frac{1}{|\psi'(x)|}\right)^\frac{p-1}{p},
\end{displaymath}
noting that $\psi' \neq 0$ almost everywhere on $\mathcal{G}_+$. Rearranging,
\begin{displaymath}
	\sum_{x \in \mathcal{G}(t)}|\psi'(x)|^{p-1} \geq \nu(t)^p \left(\sum_{x\in\mathcal{G}(t)}\frac{1}{|\psi'(x)|}\right)^{-(p-1)}
	\geq \eta^p \left(\sum_{x \in \mathcal{G}(t)} \frac{1}{|\psi'(x)|}\right)^{-(p-1)}.
\end{displaymath}
Hence
\begin{displaymath}
\begin{split}
	\int_{\mathcal{G}_\excset} |\psi'|^p\,dx &= \int_0^{\max \excset} \sum_{x \in \mathcal{G}(t)} |\psi'(x)|^{p-1}\,dt
	\geq \eta^p \int_0^{\max \excset} \left(\sum_{x \in \mathcal{G}(t)} \frac{1}{|\psi'(x)|}\right)^{-(p-1)}\,dt\\
	&= \eta^p \int_0^{\max \excset} \left(\frac{\eta}{|(\psi^\ast)'(y_t)|}\right)^{-(p-1)}\,dt
	= \eta \int_0^{\max \excset} |(\psi^\ast)'(y_t)|^{p-1}\,dt.
\end{split}
\end{displaymath}
Another application of the coarea formula yields
\begin{displaymath}
	\eta \int_0^{\max \excset} |(\psi^\ast)'(y_t)|^{p-1}\,dt = \int_{\bigcup_{i=1}^\eta \tilde e_i}|(\psi^\ast)'(x)|^p\,dx,
\end{displaymath}
recalling that $\psi^\ast$ is equal on each edge $\tilde e_i$. This completes the proof of \eqref{eq:p-symm-part}. We have now shown that
\begin{displaymath}
	\mu = \frac{\int_{\mathcal{G}_+} |\psi'|^p\,dx}{\int_{\mathcal{G}_+} |\psi|^p\,dx}
	\geq \frac{\int_{\mathcal{G}_+^\ast} |(\psi^\ast)'|^p\,dx}{\int_{\mathcal{G}_+^\ast}|\psi^\ast|^p\,dx}
        \geq \eigI_1^{(p)} (\mathcal{G}_+^\ast),
\end{displaymath}
where $\eigI_1^{(p)} (\mathcal{G}_+^\ast)$ is the first $p$-Laplacian eigenvalue of $\mathcal{G}_+^\ast$ with the natural condition at the central vertex $\mv_0$ and Dirichlet conditions at the other vertices $\tilde \mv_1,\ldots,\tilde \mv_\eta$ (cf.~\eqref{eq:dirichlet-rayleigh-quotient-p}).

If $\eta = 2$, then $\mv_0$ has degree two and $\eigI_1^{(p)} (\mathcal{G}_+^\ast)$ is the first eigenvalue of the Dirichlet $p$-Laplacian on an interval of length $2\cdot |\mathcal{G}_\excset|/\eta = |\mathcal{G}_\excset|=|\mathcal{G}_+| \leq L/2$, which is obviously in turn no smaller than the smallest nontrivial $p$-Laplacian eigenvalue of a loop of length $L$, or equivalently, of an interval of length
\begin{displaymath}
	\frac{L}{2} = \frac{L+\ell_{\max}(\eta-2)}{\eta}, \qquad \eta=2,
\end{displaymath}
with Neumann conditions at its endpoints. Since the lowest nonzero eigenvalue of the $p$-Laplacian with Neumann boundary conditions on an interval of length $\ell$ is in fact
\begin{displaymath}
(p-1) \left(\frac{\pi_p}{\ell}\right)^p
\end{displaymath}
(see, e.g., \cite[Thm.~3.4]{LanEdm11}), this proves \eqref{eq:connectp2}.

If $\eta \geq 3$, we need to account for the possibility of loops $\hat e_1,\ldots,\hat e_n$, $n \geq 0$, at $\mv_0$. To this end, we first suppose that $\hat e_1$ is the longest loop at $\mv_0$ and lengthen the others if necessary to make them all exactly as long.  We next ``split apart'' each loop in its midpoint, i.e., we create a new graph with no loops but $2n$ pendant edges at $\mv_0$ of equal length $|\hat e_1|/2$, all having the Neumann condition at their free endpoint. Recalling that $2n < \eta$, we next attach $\eta - 2n$ additional pendant edges of length $|\hat e_1|/2$ to $\mv_0$, again with Neumann conditions at the free ends.
Easy variational arguments similar to the ones in \cite[Thm 3.1.10]{BerKuc_graphs} or \cite[Lemma~2.3]{K2M2_ahp16} show that every step of this process can only decrease $\eigI_1^{(p)} (\mathcal{G}_+^\ast)$.

We claim that the resulting graph, call it $\widehat{\mathcal{G}}_+$, see
Fig.~\ref{fig:stower}(b), has first eigenvalue $\eigI_1^{(p)}(\widehat{\mathcal{G}}_+)$ no smaller than the smallest nontrivial Neumann $p$-Laplace eigenvalue of an interval of length
\begin{displaymath}
	\frac{L+(\eta-2)|\hat e_1|}{\eta}.
\end{displaymath}
Since there must have been an edge in the original graph $\mathcal{G}$ at least as long as $\hat e_1$, we have $|\hat e_1| \leq \ell_{\max}$. Since $\eigO_1^{(p)}$ on an interval is a decreasing function of the interval's length and since we have shown $\eigO_1^{(p)}(\mathcal{G}) \geq \eigI_1^{(p)}(\mathcal{G}_+)$, this claim will complete the proof of the theorem.

To prove the claim, we observe that $\widehat{\mathcal{G}}_+$ consists of $\eta$ identical copies of an interval of length $|\hat e_1|/2 + \mathcal{G}_\excset / \eta$, where the vertex conditions corresponding to $\eigI_1^{(p)}(\mathcal{G}_+)$ consist of a Neumann condition at one end and a Dirichlet one at the other of each interval; all these copies are joined together at the common vertex $\mv_0$ at a distance $|\hat e_1|/2$ from the Neumann end, as in Fig.~\ref{fig:stower}(b). The obvious bijection between (nonnegative) eigenfunctions on the two graphs shows that the first eigenvalue of $\widehat{G}_+$ is equal to the
first eigenvalue of any of these intervals, which each have length
\begin{displaymath}
  \frac{|\hat e_1|}{2} + \frac{|\mathcal{G}_\excset|}{\eta} 
  \leq \frac{|\hat e_1|}{2} + \left(\frac{L}{2}-|\hat e_1|\right)\frac{1}{\eta}
  = \frac{L+(\eta-2)|\hat e_1|}{2\eta},
\end{displaymath}
since $|\mathcal{G}_\excset| \leq |\mathcal{G}_+| - |\hat e_1|$ and $|\mathcal{G}_+| \leq L/2$ by assumption. But the first eigenvalue of the $p$-Laplacian on an interval of length $(L+(\eta-2)\ell_{\max})/(2\eta)$ with a Dirichlet condition at one end and a Neumann one at the other is equal to the smallest nontrivial Neumann eigenvalue of an interval of twice the length, $(L+(\eta-2)\ell_{\max})/\eta$. This proves the claim.
\end{proof}

\begin{remark}[The case of equality]
\label{rem:equality}
(a) Equality in \eqref{eq:connectp2} holds if and only if $\mathcal{G}$ is a finite chain of cycles -- more precisely, in the terminology of Section~\ref{sec:comb-graph}, a $2$-pumpkin chain, where the two edges in each pumpkin are as long as each other, and where each of the two terminal vertices may have a loop attached to it. In \cite{BanLev_prep16}, these are called \emph{symmetric necklace graphs} and identified as the corresponding minimizers when $p=2$; they are exactly those graphs having the same first nontrivial eigenvalue as a loop, among all graphs of edge connectivity at least two with a given total length. (In particular, one may show that $\eigO_1^{(p)}$ is in fact an eigenvalue on such graphs.)

To see this, first observe that equality implies $|\mathcal{G}_+|=L/2$. 
Then also $|\mathcal{G}_-|=L/2$, because if the eigenfunction were 
to vanish on a set of positive length, we could make $\mathcal{G}$ 
strictly smaller without changing the eigenvalue.
Moreover, since
\begin{displaymath}
	\frac{\int_{\mathcal{G}_+} |\psi'|^p\,dx}{\int_{\mathcal{G}_+} |\psi|^p\,dx}
	= \frac{\int_{\mathcal{G}_+^\ast} |(\psi^\ast)'|^p\,dx}{\int_{\mathcal{G}_+^\ast}|\psi^\ast|^p\,dx},
\end{displaymath}
we must have $\nu(t)=\eta=2$ for almost all $t\in [0,\max \psi]$.
Thus up to a finite exceptional set, $\mathcal{G}$ must
consist of two paths representing the preimages of the set $(\min
\psi,\max\psi)$. This is only possible if $\mathcal{G}$ is a cycle or
a pumpkin chain of degree two (with loops at each end under the
convention that no vertex may have degree two).
 That the pumpkin chain
must be regular can be shown by noting that to have equality in
H\"older's inequality, we need $|\psi'(x)| = |\psi'(y)|$ for any $x$ and $y$
in the same $\lvlsf(t)$.  Therefore, the function $\psi$
is identical along the two paths, and, in particular, the paths have
the same length.

(b) Equality in \eqref{eq:connectn} never holds. This follows since we always have $\eigI_1^{(p)} (\mathcal{G}_+^\ast) > \eigI_1^{(p)}(\widehat{\mathcal{G}}_+)$ if $\eta \geq 3$: equality would require $\mathcal{G}_+^\ast$ to have exactly $\eta/2$ loops of length $|\hat e_1|$ at $\mv_0$, but we already observed that $n < \eta/2$. However, equality for fixed $\eta \geq 3$ and $L>0$ is attained in the asymptotic limit for some classes of graphs as $\ell_{\max} \to 0$ (and hence the number of edges tends to infinity).  To show this, it is enough to construct a sequence of graphs with constant $\eigO_1^{(p)}(\mathcal{G}_n)$ but $\ell_{\max}(\mathcal{G}_n) \to 0$.  This is achieved, for example, by $\eta$-regular pumpkin chains with $n$ pumpkins and edges of equal length (which works out to be $L/n\eta$).

(c) If we happen to know that $\eigO_1^{(p)}(\mathcal{G})$ has an eigenfunction which attains its maximum and minimum on vertices of $\mathcal{G}$, then the proof of Theorem~\ref{thm:connectn} in fact yields
\begin{displaymath}
	\eigO_1^{(p)}(\mathcal{G}) \geq \left(\frac{\eta \pi_p}{L}\right)^p\ ;
\end{displaymath}
in particular, $\eigO_1(\mathcal{G}) \geq \pi^2 \eta^2/L^2$. This follows because in this case $\mathcal{G}_+^\ast$ does not have any loops.
\end{remark}

\begin{remark}
As mentioned in the introduction, Fiedler derived both a lower and an upper estimate on the spectral gap of the discrete Laplacian on combinatorial graphs based on the edge connectivity $\eta$. We are not aware of upper estimates on the spectral gap of the differential Laplacian on quantum graphs based on $\eta$, and indeed this does not appear to be a natural quantity for upper bounds. Indeed, a small perturbation of a graph with large $\eta$ can greatly decrease $\eta$ without significantly affecting the eigenvalues. Consider, for example, a sequence of ``pumpkin dumbbells'' $\Graph_n$ -- each consisting of two identical pumpkins, both having $k_n$ edges, joined by a single, short edge (a ``handle''). Then we can let $k_n \to \infty$ while keeping the total length fixed  $L>0$, by shrinking the edges correspondingly. If the handle also shrinks, then we can expect $\mu_1^{(p)}(\Graph_n) \to \infty$ (for $p=2$ this follows from \cite[Thm.~7.2]{K2M2_ahp16}, since the diameter of $\Graph_n$ tends to zero). Thus it is possible that $\eta$ and $L$ remain fixed, while $\mu_1^{(p)} \to \infty$. One could ask whether we might obtain a bound if we impose a minimum on $\ell_{\min}$, but we leave open the question of whether it is then possible to do better than trivial bounds in terms of $\ell_{\min}$ alone such as
\begin{displaymath}
	(p-1)\left(\frac{2\pi_p}{\ell_{\min}}\right)^p  ,
\end{displaymath}
corresponding to the first eigenvalue of the Dirichlet $p$-Laplacian on an interval of length $\ell_{\min}/2$.
\end{remark}

\subsection{Implications for the normalized Laplacian}
\label{rem:otherthanfiedler}

Let us mention a tangential but notable consequence of Theorem~\ref{thm:connectn}: von Below \cite{Bel_laa85} observed a transference principle for the spectra of the differential Laplacian on an \emph{equilateral} metric graph $\Graph$ and of the normalized Laplacian $\mathcal L_{\rm norm}$ on the underlying combinatorial graph $\mG$, cf.~\cite{Chung_spectralgraph}.  If all edges of $\Graph$ have the same length (without loss of generality, length 1), then the spectral gap $\eigO_1(\Graph)=\eigO_1^{(2)}(\Graph)$ of the differential Laplacian and the lowest nonzero eigenvalue
\begin{displaymath}
\alpha_1(\mG):= \inf \left\{ \frac{\sum\limits_{\me \in \mE } |f(\me_{init})-f(\me_{term})|^2 }{ \sum\limits_{\mv \in \mV}
	|f(\mv)|^2\deg (\mv)} :\, 0 \neq f \in \R^\mV,\, \sum_{\mv \in \mV} f(\mv) \deg(\mv)= 0 \right\}\in [0,2]
\end{displaymath}
 of the normalized Laplacian $\mathcal L_{\rm norm}$ satisfy the equality
\begin{equation}\label{eq:below}
\eigO_1(\Graph) = \left(\arccos (1-\alpha_1(\mG))\right)^2,
\end{equation}
provided $\alpha_1(\mG)\in [0,2)$ (the restriction $\alpha_1(\mG)\neq 2$ may actually be dropped unless $\mG$ is a non-bipartite unicyclic graph, in which case $\pi^2$ will be an eigenvalue of the differential Laplacian even if $2$ is not an eigenvalue of $\mathcal L_{\rm norm}$). By definition, $\alpha_1(\mG)$ is related to the lowest nonzero eigenvalue $\gamma_1(\mG)$ of $\mathcal L$ by
\begin{equation}\label{eq:normunnormtriv}
\alpha_1(\mG)\deg_{\rm max}(\mG) \ge \gamma_1(\mG)\ge \alpha_1(\mG)\deg_{\rm min}(\mG)\ ,
\end{equation}
where $\deg_{\min}(\mG)$ and $\deg_{\max}(\mG)$ are the minimal and maximal degree of the graph, respectively.
Thus, combining~\eqref{eq:normunnormtriv} and~\eqref{eq:discconnectn} we obtain the lower estimate
\begin{equation}\label{eq:trivial}
\alpha_1(\mG)\ge \frac{2\eta}{\deg_{\max}(\mG)}\left[ 1-\cos\left( \frac{\pi}{V}\right)\right]\ ;
\end{equation}
as an essentially trivial consequence of~\eqref{eq:connectn}; 
but apart from this, no further estimate on $\alpha_1$ based on $\eta$ seems to have appeared in the literature so far and in fact the spectral properties of these two matrices are known to be rather different -- the normalized Laplacian mirroring the geometry of a combinatorial graph in a more reliable way. 

Combining Corollary~\ref{cor:connectncor} with \eqref{eq:below} we obtain
\begin{corollary}
\label{cor:ec-cg}
Let $\mG$ be a combinatorial graph on $E$ edges with edge connectivity $\eta$. Then the lowest nonzero  eigenvalue $\alpha_1(\mG)$ of the normalized Laplacian $\mathcal L_{\rm norm}$ satisfies
\begin{equation}\label{eq:normconnectn}
\alpha_1(\mG)\ge 1-\cos\left(\frac{\pi\eta}{E+(\eta-2)_+}\right).
\end{equation}
\end{corollary}

To the best of our knowledge this is also the first instance of estimates of combinatorial spectral graph theory obtained by quantum graph methods, instead of the other way around.
It is not known whether von Below's formula~\eqref{eq:below} also extends to the case of $p$-Laplacians for $p\ne 2$ and we thus do not know whether~\eqref{eq:normconnectn} can also be extended to the case of normalized $p$-Laplacians on combinatorial graphs.

\begin{example}
\label{exa:wheels}
Using $\cos  x\simeq 1-x^2/2$ one obtains the asymptotic behavior
\begin{displaymath}
\frac{2\eta}{\deg_{\max}(\mG)}\frac{\pi^2}{2V^2}\qquad\hbox{and}\qquad \frac{\pi^2\eta^2}{2(E+(\eta-2)_+)^2}
\end{displaymath}
 as $V,E\to \infty$, respectively, for the right hand sides of~\eqref{eq:normunnormtriv} and~\eqref{eq:normconnectn}. Let us consider the case of a wheel graph $\mW_{n+1}$ on $n+1$ vertices, which has $2n$ edges, whose maximal degree is $n$ and whose discrete edge connectivity is 3. Then the lower estimates become
\begin{displaymath}
\frac{3\pi^2}{n^3}\qquad\hbox{and}\qquad \frac{9\pi^2}{2(2n+1)^2}
\end{displaymath}
which shows that our estimate is substantially better. The correct spectral gap of the wheel graph was computed in~\cite[\S~2.3.1]{But08} to be
\begin{displaymath}
\alpha_1(\mW_{n+1})=1-\frac{2}{3}\cos \frac{2\pi}{n},
\end{displaymath}
to be compared with our estimate
\begin{displaymath}
\alpha_1(\mW_{n+1})\ge 1-\cos	\frac{3\pi}{2n+1}\ .
\end{displaymath}
\end{example}

\section{Eigenvalue estimates for graphs with Dirichlet vertices}
\label{sec:qg_betti}

Our main tools in this section will be a symmetrization argument similar 
to the one in the proof of Theorem~\ref{thm:connectn} and
the following Lemma established by Band, Berkolaiko and Weyand.

\begin{lemma}[Lemma 4.2 of \cite{BanBerWey_jmp15}]
  \label{lem:interlacing_maxD}
  Let $\Graph$ be a metric graph with a vertex $v$ of degree $d$ with natural
  conditions, whose removal separates the graph into $d$ disjoint
  subgraphs.  We denote its edge set by $E_{v}$.  Let $r < d$ be a
  nonnegative integer. For a subset $E_{D}$ of $E_{v}$, with
  $\left|E_{D}\right|=r$, define $\Graph_{E_{D}}$ to be the
  modification of the graph $\Graph$ obtained by imposing the
  Dirichlet condition at $v$ for edges from $E_{D}$ and leaving the
  edges from $E_{v}\setminus E_{D}$ connected at $v$ with the natural
  condition (see Figure~\ref{fig:max_dirichlet} for an example).  Then
  \begin{equation}
    \eigI_{n-1}(\Graph)
    \leq \min_{|E_{D}|=r}\eigI_{n}(\Graph_{E_{D}})
    \leq \eigI_{n}(\Graph)
    \leq \max_{|E_{D}|=r}\eigI_{n}(\Graph_{E_{D}})
    \leq \eigI_{n+1}(\Graph).
    \label{eq:max_Dirichlet}
  \end{equation}
\end{lemma}

\begin{figure}
  \centering
  \includegraphics[scale=0.75]{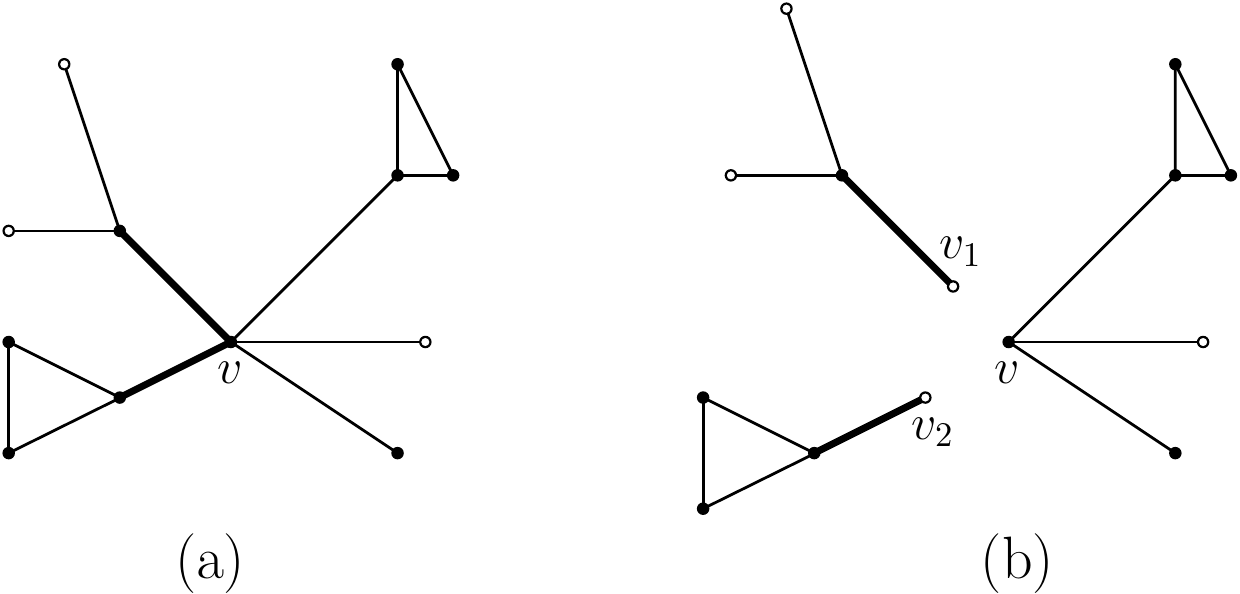}
  \caption{A graph $\Graph$ and one possibility of its modification $\Graph_{E_D}$ as in Lemma~\ref{lem:interlacing_maxD} with $r=2$.  Dirichlet vertices are shown as empty circles.}
  \label{fig:max_dirichlet}
\end{figure}

\begin{remark}
  The conditions at vertices other than $v$ can be of any type.  Also,
  the Lemma applies to the eigenvalues of Schrödinger operators
  (i.e.\ the Laplacian plus a potential).

  In Lemma 4.2 of \cite{BanBerWey_jmp15}, the condition $r \leq d$ was
  used.  In fact, one can not choose $r = d$, as Lemma~4.1 of that
  reference, which is used in the proof, requires decomposition of
  $\Graph$ at $v$ into $r+1$ disjoint subgraphs.
\end{remark}

\subsection{Lower bounds}

We will derive an estimate for the first eigenvalue of a quantum graph $\Graph$ 
with some vertices of degree one on which Dirichlet conditions are imposed.  As easy
corollaries we obtain many existing and some new results, including
for graphs without Dirichlet vertices.  Throughout Section~\ref{sec:qg_betti} we
will number the eigenvalues starting with 1, i.e. 
\begin{displaymath}
0\leq \eigIp_1(\Graph) < \eigIp_2(\Graph) \leq \ldots ,
\end{displaymath}
independent of whether the graph has Dirichlet vertices or not.\footnote{Note that 
the ground state eigenvalue $\eigIp_1$ remains simple if $\Graph$ has at least one 
Dirichlet vertex; this is part of the statement in~\cite[Thm.~3.3]{DelRos_amp16}, 
for example.}

\begin{lemma}
  \label{lemma:two_Dirichlet_vertices}
  Let $\Graph$ be a quantum graph with at least two vertices of
  degree one, and such that if we merge all degree one vertices
  together, the resulting graph has edge connectivity two or more.
  Impose Dirichlet conditions on degree one vertices and natural
  conditions elsewhere.  Then the ground state eigenvalue of $\mathcal
  G$ is bounded below by
  \begin{equation}
    \label{eq:friedlander_lemma3-p}
    \eigIp_1 (\Graph) \geq \left(\frac{\pi_p}{L}\right)^p.
  \end{equation}
  The minimizers are $2$-regular pumpkin chains with two edges of
  equal length attached to one of the endpoints and the degenerate case of an interval with two
  Dirichlet endpoints (it is ``degenerate'' since it can be seen as having an
  infinitesimal pumpkin chain at the midpoint of the interval).
\end{lemma}

\begin{proof}
  Consider the ground state eigenfunction $ \psi_1 $, which
  by~\cite[Thm.~3.3]{DelRos_amp16} can be chosen nonnegative. Let us
  denote by $ M $ its maximum and by $ x_0 $ any point on $ \Graph $
  where the maximum is attained. There exist two edge-disjoint paths
  connecting $ x_0 $ to the two Dirichlet vertices, since otherwise
  the graph obtained after the deletion of the edges with Dirichlet
  conditions would have discrete edge connectivity less than two. The
  usual symmetrization technique maps the function $ \psi_1 $ to a
  function on the interval $ [0, L]$ satisfying the Dirichlet
  condition at point $ 0. $ The number of preimages $ \nu (t) $ is at
  least two almost everywhere. Hence the Rayleigh quotient is at least
  $2^p $ times the lowest eigenvalue of the Laplacian on $ [0,L ] $
  with Dirichlet and Neumann boundary conditions. The latter
  eigenvalue is $ \left( \frac{\pi_p}{2 L} \right)^p. $ Hence formula
  (\ref{eq:friedlander_lemma3-p}) follows.

  The minimizing graph is not unique.  Just as in
  Remark~\ref{rem:equality}(a), the symmetrization process yields an
  equality if and only if the number of pre-images $ \nu(t) $ is
  precisely equal to two up to a finite exceptional set.  This
  requires that $\Graph$ be a $2$-regular pumpkin chain with two edges
  of equal length (the Dirichlet edges) attached to one of the end
  points. But such graphs are easily seen to have first eigenvalue
  equal to $(\pi_p/L)^p$, i.e., that of an interval of length $L$ with
  Dirichlet boundary conditions. Indeed, the unique ground state on
  the interval is known to be given by $\sin_p$, and the graph
  $ \Graph $ can in this case be obtained from the interval by joining
  together points on the interval where $\sin_p$ attains equal
  values. This preserves the Rayleigh quotient, and so there is
  equality in \eqref{eq:friedlander_lemma3-p}.
\end{proof}

As a first corollary we obtain a generalization of an estimate proved
in the case $p=2$ by Nicaise~\cite[Th\'eo.~3.1]{Nic_bsm87} (see also
Firedlander~\cite[Lemma 3]{Fri_aif05}).
  
\begin{corollary}
  \label{cor:Nic_Fri}
  Let $\Graph$ have at least one Dirichlet vertex.  Then
  \begin{equation}
    \label{eq:Nic_Fri}
    \eigIp_1(\Graph) \geq \left(\frac{\pi_p}{2L}\right)^p.
  \end{equation}
\end{corollary}

\begin{proof}
Consider the graph $ \Graph^2 $ obtained from $ \Graph $ by doubling all the edges -- substituting every edge in $ \Graph $ with two parallel
edges of the same length as before and keeping the vertices. Every eigenfunction for $ \Graph $ can be extended to an eigenfunction on $ \Graph^2 $ assigning the
same values on parallel edges. It follows that $ \lambda_n (\Graph^2) \leq \lambda_n (\Graph). $ The graph $ \Graph^2 $
has at least one Dirichlet vertex of degree two, but this vertex can be split into two Dirichlet vertices of degree one. Then
estimate (\ref{eq:friedlander_lemma3-p}) implies (\ref{eq:Nic_Fri}) if one takes into account that the length of $ \Graph^2 $ is twice the length of $ \Graph $.
\end{proof}

When all vertex conditions are Neumann,
Lemma~\ref{lemma:two_Dirichlet_vertices} and its
Corollary~\ref{cor:Nic_Fri} yield a result of Nicaise \cite{Nic_bsm87}
(see also \cite{Fri_aif05}) and a result of Band and Levy \cite{BanLev_prep16}
which is also the special case of Theorem~\ref{thm:connectn} with
$p=2$ and $\eta=2$.

\begin{corollary}
  \label{cor:connect}
  \begin{enumerate}
  \item (\cite[Th\'eo.~3.1]{Nic_bsm87}, no Dirichlet vertices case;
    \cite[Thm~1]{Fri_aif05}, case $j=2$) Let $\Graph$ be a connected graph
    with natural conditions at every vertex.  Then its first
    nontrivial eigenvalue satisfies
    \begin{equation}
      \label{eq:lower_bound_conn1}
      \eigI_2(\Graph) \geq \frac{\pi^2}{L^2} 
    \end{equation}
  \item (\cite[Thm 2.1]{BanLev_prep16}, see also \cite{KurNab_jst14}) If, in addition, $\Graph$ has
    connectivity $\eta \geq 2$, then its smallest nontrivial eigenvalue
    satisfies
    \begin{equation}
      \label{eq:lower_bound_conn2}
      \eigI_2(\Graph) \geq \frac{4\pi^2}{L^2}
    \end{equation}
  \end{enumerate}
\end{corollary}

\begin{proof}
  The eigenfunction corresponding to the smallest nontrivial eigenvalue
  has at least two nodal domains\footnote{In the generic case of a
  simple eigenvalue and the corresponding eigenfunction not vanishing on
  the vertices, the number of nodal domains is exactly two by Courant's bound.} 
  since the eigenfunction must be orthogonal to the constant
  function. Apply Corollary~\ref{cor:Nic_Fri} to the nodal domain
  considered as a graph of the total length $\leq L/2$ to obtain
  equation~\eqref{eq:lower_bound_conn1}.

  If the graph has connectivity $\eta \geq 2$, it is easy to see that
  a nodal domain satisfies the conditions of
  Lemma~\ref{lemma:two_Dirichlet_vertices}.  We now apply the Lemma to
  the nodal domain of the total length $\leq L/2$ to obtain
  equation~\eqref{eq:lower_bound_conn2}.
\end{proof}

Finally, we consider the special case of a tree graph with Dirichlet
conditions on some leaves (i.e.\ vertices of degree one). Here we
restrict ourselves to the case $p=2$. To the best of the authors'
knowledge, no analogue of this Lemma has previously appeared in the
literature.

\begin{lemma}
  \label{lem:lower_bounds_trees}
  Let $\Graph$ be a tree with Dirichlet conditions imposed at the
  leaves and natural conditions elsewhere.  Then
  \begin{equation}
    \label{eq:lower_bound_trees}
    \eigI_1(\Graph) \geq \frac{\pi^2}{(\diam{\Graph})^2},
  \end{equation}
  where $\diam{\Graph}$ is the diameter of the graph (the maximal distance between a
  pair of leaves).

  If $\Graph$ has one Neumann leaf and all other leaves Dirichlet, the
  bound is
  \begin{equation}
    \label{eq:lower_bound_trees_1N}
    \eigI_1(\Graph) \geq \frac{\pi^2}{4(\diam{\Graph})^2},
  \end{equation}
\end{lemma}

\begin{proof}
  We repeatedly apply Lemma~\ref{lem:interlacing_maxD} (second
  inequality) with $r=1$ at vertices of degree three or more, choosing
  the graph with the minimal $\eigI_1(\Graph_{E_D})$ at every step.
  We stop when there are no vertices of degree larger than $2$ and we
  absorb all vertices of degree two into the edges.  The graph
  is thus reduced to a collection of disjoint intervals with
  Dirichlet conditions, and the first eigenvalue comes from the
  longest of them.  The longest interval possible is the path giving
  the diameter of the graph.

  If the tree has one Neumann leaf, we double the tree and reflect its
  eigenfunction across this leaf to obtain a tree with all leaves
  Dirichlet and the diameter less than or equal to $2\diam{\Graph}$.
\end{proof}

\subsection{Lower bounds for all eigenvalues}

We are now in position to obtain an improved version of the lower estimate due
to Friedlander \cite{Fri_aif05} on \emph{all} eigenvalues of a quantum graph.
Here and in the rest of Section~\ref{sec:qg_betti} we shall refrain from considering the $p$-Laplacian for
$p\ne 2$, as the theory of higher eigenvalues for these nonlinear operators is
rather technical and goes beyond the scope of this article.

\begin{theorem}
  \label{thm:low_bound_all_eig}
  Let $\Graph$ be a quantum graph with $|N|\geq 0$ vertices of degree one with the Neumann
  condition (and any number $|D|\geq 0$ of vertices with the Dirichlet condition).
  Assume that $\Graph$ is not a cycle. Then for all $k \geq 2$
 \begin{equation}
    \label{eq:low_bound_all_eig}
    \eigI_k(\Graph) \geq 
    \begin{cases}
      \left(k - \frac{|N|+\beta}2 \right)^2 \dfrac{\pi^2}{L^2}  
      & \mbox{if } k \geq |N| + \beta\\[10pt]
      \dfrac{k^2 \pi^2}{4 L^2} & \mbox{otherwise},
    \end{cases}
  \end{equation}
  where
  \begin{displaymath}
    \beta := E - V + 1 \geq 0
  \end{displaymath}
  is the first Betti number of the graph.
  If there is at least one vertex with the Dirichlet condition, then we may replace the assumption $k\geq 2$ by $k\geq 1$.
\end{theorem}

\begin{remark}
  Friedlander \cite[Thm.~1]{Fri_aif05} proved the estimate $\eigI_k(\Graph) \geq
  \frac{k^2 \pi^2}{4 L^2}$ for all $k$, but it is easy to see that the first
  expression in \eqref{eq:low_bound_all_eig} gives a better bound for all $k \geq |N|
  + \beta$. Theorem~\ref{thm:low_bound_all_eig} has the added advantage of being asymptotically sharp
  in the sense that both the eigenvalue and its bound have the same
  asymptotic form $\frac{k^2\pi^2}{L^2}+ o(1)$ as $k\to\infty$.
\end{remark}

\begin{proof}
  If $\Graph$ is not a tree, we find an edge whose removal would not
  disconnect the graph.  Let $\mv$ be a vertex to which this edge is
  incident; since $\Graph$ is not a cycle, without loss of generality
  we can assume its degree is 3 or larger (otherwise this vertex can
  be absorbed into the edge).  We disconnect the edge from this
  vertex, reducing $\beta$ by one and creating an extra vertex of
  degree one where we impose the Neumann condition, see
  Figure~\ref{fig:beta_break}. We keep natural conditions at
  $\mv$. Then the new graph is not a cycle, as a new vertex of degree
  $1$ was created.  We may therefore repeat the process inductively
  until we obtain a tree $\Tree$ with $|N'| = |N|+\beta$ Neumann
  vertices.

  Since the eigenvalues are reduced at every step, $\eigI_k(\Graph) \geq
  \eigI_k(\Tree)$.  It is therefore enough to verify the inequality for trees.

  \begin{figure}
    \centering
    \includegraphics{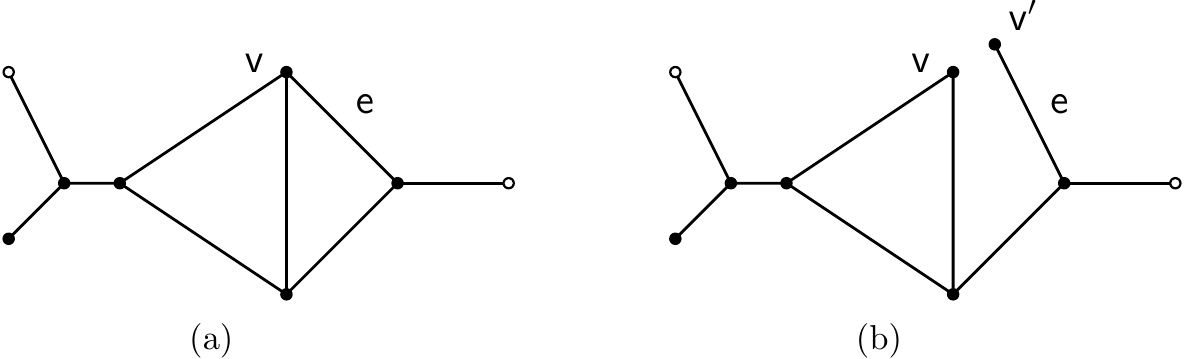}
    \caption{Disconnecting the edge $\me$ from the vertex $\mv$ in the proof of
      Theorem~\ref{thm:low_bound_all_eig}. This operation reduces $\beta$ by 1
      at the expense of increasing the number of Neumann vertices by 1.}
    \label{fig:beta_break}
  \end{figure}

  Given a tree $\Graph$ we can find an arbitrarily small perturbation under which 
  the $k$-th eigenvalue is simple and its
  eigenfunction is nonzero on vertices \cite{BerLiu_jmaa17}.  In these
  circumstances the $k$-th eigenfunction has exactly $k$ nodal domains
  \cite{PokPryObe_mz96,Sch_wrcm06} (see also \cite[Thm.~6.4]{BerKuc_incol12} for
  a short proof).  Each nodal domain is a subtree $\Tree_j$, and with vertex
  conditions inherited from $\Graph$ (plus Dirichlet conditions on the nodal
  domain boundaries), $\eigI_k(\Graph)$ is the first eigenvalue of the subtree.

  There are at most $|N|$ subtrees with some Neumann conditions on their leaves.
  Since these are nodal subtrees ($k>1$), there are also some leaves with
  Dirichlet conditions and we can use estimate \eqref{eq:Nic_Fri} in the form
  $L_j \sqrt{\eigI} \geq \pi/2$.  The same conclusion is true if $k=1$ and $\Graph$ 
  has at least one Dirichlet vertex.

  If $k \geq |N|$, we also have at least
  $k-|N|$ subtrees with \emph{only} Dirichlet conditions at the leaves, for
  which we can use the bound of Theorem~\ref{lem:lower_bounds_trees} but with
  the diameter substituted by the total length of the subtree, i.e.\ $L_j
  \sqrt\eigI \geq (\diam{\Tree}_j)\sqrt\eigI \geq \pi$.  We therefore have
  \begin{equation*}
    L \sqrt{\eigI_k(\Graph)} 
    = \sum_{j=1}^k L_j \sqrt{\eigI_1(\Tree_j)}
    \geq |N| \frac{\pi}{2} + (k-|N|) \pi = \left(k -
      \tfrac{|N|}{2}\right) \pi.\
  \end{equation*}
  
  When $k < |N|$, we use estimate
  \eqref{eq:Nic_Fri} for each of the $k$ nodal subtrees, obtaining
  Friedlander's bound.
\end{proof}

\subsection{Upper bound for all eigenvalues}

We now slightly extend and improve the beautiful recent result of
Ariturk \cite[Thm.~1.2]{Ari_prep16}, providing at the same time a simpler
proof.

\begin{theorem}
  \label{thm:upp_bound_all_eig}
  Let $\Graph$ be a connected quantum graph with Dirichlet or Neumann
  conditions at the vertices of degree one and natural condition
  elsewhere.  If $\Graph$ is not a cycle, then for all
  $k\in \mathbb N$
 \begin{equation}
    \label{eq:upp_bound_all_eig}
    \eigI_k (\Graph)
    \leq \left(k - 2 + \beta + |D| + \tfrac{|N|+\beta}2\right)^2 \frac{\pi^2}{L^2},
  \end{equation}
  where the set of Dirichlet vertices is denoted by $D$ and the set of Neumann
  vertices of degree one is denoted by $N$.
\end{theorem}

\begin{remark}\label{rem:gregarit}
  (a) Both Theorem~\ref{thm:low_bound_all_eig} and
  Theorem~\ref{thm:upp_bound_all_eig} are false on cycle graphs. If
  $\Graph$ is a loop with a pendant edge attached to it (i.e., a
  ``lollipop'' or ``lasso'' graph), and with the Neumann condition at
  the pendant vertex, then Theorems~\ref{thm:low_bound_all_eig}
  and~\ref{thm:upp_bound_all_eig} result in the two-sided bound
  \begin{displaymath}
	\frac{(k-1)^2\pi^2}{L^2} \leq \lambda_k (\mathcal{G})\leq 
	\frac{k^2 \pi^2}{L^2}
  \end{displaymath}
  for $k \geq 2$ large enough. Remarkably, as the length of the
  pendant edge converges to zero, odd-numbered eigenvalues converge to
  the lower bound while even-numbered ones converge to the upper bound.

  (b) Similarly to the bound of Theorem~\ref{thm:low_bound_all_eig}, the upper bound
  in \eqref{eq:upp_bound_all_eig} is asymptotically sharp.  The only other
  asymptotically sharp bounds known to the authors is the one by Nicaise
  \cite[Th. 2.4]{Nic_bsm87}, who proved that for \emph{equilateral} graphs
  \begin{equation}
    \label{eq:Nic_equi}
    \frac{(k-1-L)^2\pi^2}{L^2}\le \eigI_k(\Graph)
    \le \frac{(k-1+L)^2\pi^2}{L^2}\ .    
  \end{equation}
  Under some circumstances (for example for complete graphs with many
  vertices), bound \eqref{eq:Nic_equi} is tighter than
  \eqref{eq:upp_bound_all_eig}; informal scaling arguments suggest
  that, for large $\beta$, the factor in front of $\beta$ in
  \eqref{eq:upp_bound_all_eig} should be 1 instead of the current
  $\frac{3}{2}$ (see also the discussion following Theorem 1.2 in
  \cite{Ari_prep16}; note that in the results of \cite{Ari_prep16} the
  corresponding factor is 2).
\end{remark}

\begin{proof}[Proof of Theorem~\ref{thm:upp_bound_all_eig}]
  If $\Graph$ is not a tree (i.e.\ if $\beta > 0$) and not a cycle, we repeat the 
  process described at the beginning of the
  proof of Theorem~\ref{thm:low_bound_all_eig}, disconnecting $\beta$ edges at
  vertices and creating a tree $\Tree$ with $\beta$ additional Neumann vertices
  of degree one.  At every step, the eigenvalue goes down, but not further than
  the next eigenvalue (see \cite[Thm 3.1.10]{BerKuc_graphs}), therefore we have
  $\eigI_k(\Graph) \leq \eigI_{k+\beta}(\Tree)$ and the bound for general
  graphs follows from the bound for trees, $\beta=0$.

  We will prove the result assuming $\Graph$ is a tree by induction on
  the number of edges.  The inequality turns into equality for a
  single edge with either Dirichlet or Neumann or mixed conditions.

  Choose an arbitrary vertex $\mv$ of degree three or
  more and apply Lemma~\ref{lem:interlacing_maxD} (the third
  inequality) with $r=1$.  Let $\Graph'$ be the graph realizing this
  inequality; it is a disjoint union of two trees, denote them by
  $\Tree_1$ and $\Tree_2$.  Without loss of generality,
  $\eigI_k(\Graph')$ is an eigenvalue of $\Tree_1$; denote its
  position in the spectrum of $\Tree_1$ by $j \leq k$.  We therefore
  have
  \begin{equation*}
    \eigI_k(\Graph') = \eigI_j(\Tree_1) 
    \qquad \mbox{and} \qquad
    \eigI_k(\Graph') \leq \eigI_{k-j+1}(\Tree_2).
  \end{equation*}
  Denoting by $L_1$ and $L_2$ the total lengths of the two trees, we
  have $L = L_1 + L_2$.  Denoting by $D_1$ and $D_2$ the Dirichlet
  vertices, we also have $|D| = |D_1| + |D_2| - 1$, since one
  Dirichlet vertex was added in the process of application of
  Lemma~\ref{lem:interlacing_maxD} with $r=1$.  We now use the
  inductive hypothesis for the two trees $T_1$ and $T_2$ to get
  \begin{align*}
    L \sqrt{\eigI_k(\Graph)} 
    &\leq L_1 \sqrt{\eigI_j(\Tree_1)} + L_2
    \sqrt{\eigI_{k-j+1}(\Tree_2)} \\
    &\leq \pi \left(j - 2 + |D_1| + \tfrac{|N_1|}2\right) 
    + \pi \left(k - j + 1 - 2 + |D_2| + \tfrac{|N_2|}2\right) \\
    &= \pi \left(k - 3 + |D| + 1 + \tfrac{|N|}2\right).
  \end{align*}
  This completes the proof.
\end{proof}

\subsection{Implications for the normalized and discrete Laplacians}
\label{rem:otherthanfiedler-higher}

It is known that formula~\eqref{eq:below} can be generalized to the higher eigenvalues: indeed, von Below has shown that 
\begin{equation}\label{eq:highervonbelow-2}
\mu(\Graph) = \left(\arccos (1-\alpha(\mG))\right)^2,
\end{equation}
cf.~\cite[Thm.\ on p.\ 320]{Bel_laa85},
for all eigenvalues $0=\mu_0(\Graph)<\mu_1(\Graph)\le \ldots $ of $\Delta$ within the interval $[0,\pi^2]$  and $0=\alpha_0(\mG)<\alpha_1(\mG)\le\ldots \le\alpha_{V-1}(\mG)$ of $\mathcal L_{\rm norm}$ within the interval $[0,2]$ (with the exception of $\mu=\pi^2$, which is an eigenvalue of $\Delta$ whenever $\mG$ is a non-bipartite unicyclic graph although $2$ is not an eigenvalue of $\mathcal L_{\rm norm}$); also, the multiplicities coincide, so that in fact switching to the notation adopted in this section we can write
\begin{equation}\label{eq:highervonbelow-3}
\alpha_{k-1}(\mG)=1-\cos \sqrt{\lambda_k(\Graph)}\qquad \hbox{for all }k=1,\ldots,V\ .
\end{equation}
Accordingly, all of our results in this section can be translated into both upper and lower estimates on~\emph{all} eigenvalues $\alpha_k(\mG)$ of normalized Laplacians and we obtain the following.

\begin{proposition}
\label{prop:discrete-greg}
Let $\mG$ be a connected combinatorial graph without loops, which has $V^* \geq 0$ vertices of degree $>1$. Then the eigenvalues $\alpha_k(\mG)$ of the normalized Laplacian on $\mG$ satisfy, for all $k=2,\ldots,V$, the estimates
\begin{equation}
\label{eq:condition-estimatediscr-0-concr}
1-\cos\left( \frac{2k-E+V^*-1}{2E}\cdot \pi \right) \le \alpha_{k-1}(\mG),
\end{equation}
as well as
\begin{equation}
\label{eq:condition-estimatediscr-fried}
1-\cos \frac{k\pi}{2E} \le \alpha_{k-1}(\mG),
\end{equation}
and
\begin{equation}
\label{eq:condition-estimatediscr-concr}
\alpha_{k-1}(\mG)\le 1-\cos\left(\frac{2k+3E-V-V^*-1}{2E}\cdot \pi \right),
\end{equation}
where \eqref{eq:condition-estimatediscr-0-concr} holds provided
\begin{equation}
\label{eq:condition-estimatediscr-0}
k\ge E-V^*+1
\end{equation}
and~\eqref{eq:condition-estimatediscr-concr} provided
\begin{equation}
\label{eq:condition-estimatediscr}
2k-V-V^*-1+E\le 0.
\end{equation}
\end{proposition}

Observe that conditions~\eqref{eq:condition-estimatediscr-0} and~\eqref{eq:condition-estimatediscr} imply in particular that the arguments of $\cos(\,\cdot\,)$ in~\eqref{eq:condition-estimatediscr-0-concr} and~\eqref{eq:condition-estimatediscr-concr} lie in $[0,\pi]$, the domain of monotonicity of $1-\cos(\,\cdot\,)$. 

\begin{example}
A necessary condition for the lower estimate to be nontrivial (i.e., to hold for at least one $k$) is that $V+V^*\ge E+1$: this is indeed satisfied for complete graphs $\mK_n$ with $n\le 4$ (since then $V+V^*=2n\ge \frac{n(n-1)}{2}+1 =E+1$), cycle graphs $\mC_n$ ($V+V^*=2n\ge n+1=E+1$), path graphs $\mP_n$ ($V+V^*=2n-2\ge n=E+1$), star graphs $\mS_{n+1}$ ($V+V^*=n+2\ge n+1=E+1$), wheel graphs $\mW_ {n+1}$ ($V+V^*=2n+2\ge 2n+1=E+1$, but for no complete bipartite graphs $\mK_{n,m}$ with $n\ne 1\ne m$ ($V+V^*=n+m\not\ge nm+1 =E+1$) and no nontrivial hypercube graphs $\mQ_n$ ($V+V^*=2^{n+1}\not\ge 2^{n-1}n+1 =E+1$).

In the case of star graphs $\mS_{n+1}$ our lower estimate is sharp: it reads
\begin{displaymath}
1 \le \alpha_{n-1}(\mS_{n+1}),
\end{displaymath}
which delivers the actual value; whereas in the case of the wheel graph we find
\begin{displaymath}
0.293\simeq 1-\cos \frac{\pi}{4} \le \alpha_{n-1}(\mW_{n+1}),
\end{displaymath}
to be compared with the actual value of $1-\frac{2}{3}\cos \frac{2\pi(n-1)}{n}\to \frac{1}{3}$ as $n\to \infty$.

Our \textit{upper} estimate cannot be applied unless condition~\eqref{eq:condition-estimatediscr} is satisfied: this actually fails in the case of star graphs, wheel graphs, complete graphs or nontrivial hypercube graphs, for which we cannot improve the trivial upper estimate $\alpha_k \leq 2$ for all $k$. In most cases where condition~\eqref{eq:condition-estimatediscr} does hold, like the Petersen graph (for $k=2,3$), path graphs $\mP_n$ (for $2\le k\le \frac{n}{2}$), complete bipartite graphs $\mK_{n,m}$ (for $k\le n+m-\frac12(mn-1)$) or cycle graphs $\mC_n$ (for $2\le k\le \frac{n+1}{2}$), the upper estimates are however typically much rougher than the corresponding lower estimates; this seems to substantiate our conjecture in Remark~\ref{rem:gregarit}.
\end{example}

In view of the version of~\eqref{eq:normunnormtriv} for the higher eigenvalues, namely
\begin{equation}\label{eq:normunnormtriv-high}
\alpha_k(\mG)\deg_{\rm max}(\mG) \ge \gamma_k(\mG)\ge \alpha_k(\mG)\deg_{\rm min}(\mG)\ ,
\end{equation}
analogous estimates hold also for the eigenvalues of the discrete Laplacian $\mathcal L$. We are not aware of any earlier result of this kind in the literature about discrete and normalized Laplacians, apart from the trivial estimate $\alpha_{\max}(\mG)\le 2$ and the estimates on $\gamma_{\max}(\mG)$ summarized in~\cite{LiZha98,Shi07}.

\bibliographystyle{alpha}

\begin{thebibliography}{KKMM16}

\bibitem[Ari16]{Ari_prep16}
S.~Ariturk.
\newblock Eigenvalue estimates on quantum graphs.
\newblock Preprint {\tt arXiv:1609.07471}, 2016.

\bibitem[BBW15]{BanBerWey_jmp15}
R.~Band, G.~Berkolaiko, and T.~Weyand.
\newblock Anomalous nodal count and singularities in the dispersion relation of
  honeycomb graphs.
\newblock {\em J.\ Math.\ Phys.}, 56(12):122111, 2015.

\bibitem[Bel85]{Bel_laa85}
{J.~von} Below.
\newblock A characteristic equation associated to an eigenvalue problem on
  {$c^2$}-networks.
\newblock {\em Linear Algebra Appl.}, 71:309--325, 1985.

\bibitem[Ber16]{Ber_prep16}
G.~Berkolaiko.
\newblock Elementary introduction to quantum graphs.
\newblock Preprint {\tt arXiv:1603.07356 [math-ph]}, 2016.

\bibitem[BH09]{BuhHei09}
T.~B\"uhler and M.~Hein.
\newblock Spectral clustering based on the graph $p$-{L}aplacian.
\newblock In {\em Proc.\ 26th Annual Int.\ Conf.\ Mach.\ Learning}, pages
  81--88. ACM, New York, 2009.

\bibitem[BK12]{BerKuc_incol12}
G.~Berkolaiko and P.~Kuchment.
\newblock Dependence of the spectrum of a quantum graph on vertex conditions
  and edge lengths.
\newblock In {\em Spectral Geometry}, volume~84 of {\em Proceedings of Symposia
  in Pure Mathematics}. American Math. Soc., 2012.
\newblock Preprint {\tt arXiv:1008.0369}.

\bibitem[BK13]{BerKuc_graphs}
G.~Berkolaiko and P.~Kuchment.
\newblock {\em Introduction to Quantum Graphs}, volume 186 of {\em Mathematical
  Surveys and Monographs}.
\newblock AMS, 2013.

\bibitem[BL16]{BanLev_prep16}
R.~Band and G.~L\'{e}vy.
\newblock Quantum graphs which optimize the spectral gap.
\newblock Preprint {\tt arXiv:1608.00520}, 2016.

\bibitem[BL17]{BerLiu_jmaa17}
G.~Berkolaiko and W.~Liu.
\newblock Simplicity of eigenvalues and non-vanishing of eigenfunctions of a
  quantum graph.
\newblock {\em J. Math. Anal. Appl.}, 445(1):803--818, 2017.
\newblock Preprint {\tt arXiv:1601.06225}.

\bibitem[Bol98]{Bollobas}
B.~Bollob{\'a}s.
\newblock {\em Modern graph theory}, volume 184 of {\em Graduate Texts in
  Mathematics}.
\newblock Springer-Verlag, New York, 1998.

\bibitem[BR08]{BinRyn08}
P.A.~Binding and R.P.~Rynne.
\newblock Variational and non-variational eigenvalues of the $p$-laplacian.
\newblock {\em J.\ Differential Equations}, 244:24--39, 2008.

\bibitem[But08]{But08}
S.K.~Butler.
\newblock {\em {Eigenvalues and Structures of Graphs}}.
\newblock PhD thesis, University of California, San Diego, 2008.

\bibitem[Chu97]{Chung_spectralgraph}
F.R.K.~Chung.
\newblock {\em Spectral graph theory}, volume~92 of {\em CBMS Regional
  Conference Series in Mathematics}.
\newblock Published for the Conference Board of the Mathematical Sciences,
  Washington, DC, 1997.

\bibitem[DPR16]{DelRos_amp16}
L.M.~Del~Pezzo and J.D.~Rossi.
\newblock The first eigenvalue of the {$p$}-{L}aplacian on quantum graphs.
\newblock {\em Anal. Math. Phys.}, 6(4):365--391, 2016.

\bibitem[Fie73]{Fie_cmj73}
M.~Fiedler.
\newblock Algebraic connectivity of graphs.
\newblock {\em Czechoslovak Math. J.}, 23(98):298--305, 1973.

\bibitem[Fri05]{Fri_aif05}
L.~Friedlander.
\newblock Extremal properties of eigenvalues for a metric graph.
\newblock {\em Ann. Inst. Fourier (Grenoble)}, 55(1):199--211, 2005.

\bibitem[GS06]{GnuSmi_ap06}
S.~Gnutzmann and U.~Smilansky.
\newblock Quantum graphs: Applications to quantum chaos and universal spectral
  statistics.
\newblock {\em Adv. Phys.}, 55(5--6):527--625, 2006.

\bibitem[HLM15]{hein2015mini}
M.~Hein, D.~Lenz and D.~{Mugnolo (eds.)}.
\newblock Mini-workshop: Discrete $p$-laplacians: Spectral theory and
  variational methods in mathematics and computer science.
\newblock {\em Oberwolfach Reports}, 12(1):399--447, 2015.

\bibitem[KKMM16]{K2M2_ahp16}
J.B.~Kennedy, P.~Kurasov, G.~Malenov{\'a} and D.~Mugnolo.
\newblock On the spectral gap of a quantum graph.
\newblock {\em Ann. Henri Poincar\'e}, 17(9):2439--2473, 2016.

\bibitem[KMN13]{KurMalNab_jpa13}
P.~Kurasov, G.~Malenov{\'a} and S.~Naboko.
\newblock Spectral gap for quantum graphs and their edge connectivity.
\newblock {\em J. Phys. A}, 46(27):275309, 16, 2013.

\bibitem[KN14]{KurNab_jst14}
P.~Kurasov and S.~Naboko.
\newblock Rayleigh estimates for differential operators on graphs.
\newblock {\em J. Spectr. Theory}, 4(2):211--219, 2014.

\bibitem[LE11]{LanEdm11}
J.~Lang and D.~Edmunds.
\newblock {\em Eigenvalues, embeddings and generalised trigonometric
  functions}, volume 2016 of {\em Lecture Notes in Mathematics}.
\newblock Springer, Heidelberg, 2011.

\bibitem[LZ98]{LiZha98}
J.-S.~Li and X.-D.~Zhang.
\newblock On the {L}aplacian eigenvalues of a graph.
\newblock {\em Lin.\ Algebra Appl.}, 285:305--307, 1998.

\bibitem[Mug13]{Mug_na13}
D.~Mugnolo.
\newblock Parabolic theory of the discrete {$p$}-{L}aplace operator.
\newblock {\em Nonlinear Anal.}, 87:33--60, 2013.

\bibitem[Mug14]{Mugnolo_book}
D.~Mugnolo.
\newblock {\em Semigroup methods for evolution equations on networks}.
\newblock Understanding Complex Systems. Springer, Cham, 2014.

\bibitem[Nic87]{Nic_bsm87}
S.~Nicaise.
\newblock Spectre des r{\'e}seaux topologiques finis.
\newblock {\em Bull. Sci. Math. (2)}, 111(4):401--413, 1987.

\bibitem[NY76]{NakYam76}
T.~Nakamura and M.~Yamasaki.
\newblock Generalized extremal length of an infinite network.
\newblock {\em Hiroshima Math. J.}, 6(1):95--111, 1976.

\bibitem[PPAO96]{PokPryObe_mz96}
Yu.V.~Pokorny{\u\i}, V.L.~Pryadiev and A.~Al{\cprime}-Obe{\u\i}d.
\newblock On the oscillation of the spectrum of a boundary value problem on a
  graph.
\newblock {\em Mat. Zametki}, 60(3):468--470, 1996.

\bibitem[Roh16]{Roh16_preprint}
J.~Rohleder.
\newblock Eigenvalue estimates for the laplacian on a metric tree.
\newblock {\em Proc. \ Amer. \ Math. \ Soc.}, 2016.
\newblock To appear.

\bibitem[Sch06]{Sch_wrcm06}
P.~Schapotschnikow.
\newblock Eigenvalue and nodal properties on quantum graph trees.
\newblock {\em Waves Random Complex Media}, 16(3):167--178, 2006.

\bibitem[Shi07]{Shi07}
L.~Shi.
\newblock Bounds on the ({L}aplacian) spectral radius of graphs.
\newblock {\em Lin.\ Algebra Appl.}, 422:755--770, 2007.

\bibitem[Spi09]{Spi09}
D.~Spielman.
\newblock Spectral graph theory --- {M}anuscript of a {Y}ale course (appl.\
  math.\ 561/comp.\ sci.\ 662).
\newblock Available at
  \url{http://www.cs.yale.edu/homes/spielman/561/2009/lect02-09.pdf}, 2009.

\end{thebibliography}
\def\cprime{$'$} \def\cprime{$'$} \def\cprime{$'$} \def\cprime{$'$}
  \def\cprime{$'$} \def\cprime{$'$} \def\cprime{$'$}
  \def\polhk#1{\setbox0=\hbox{#1}{\ooalign{\hidewidth
  \lower1.5ex\hbox{`}\hidewidth\crcr\unhbox0}}} \def\cprime{$'$}
  \def\cprime{$'$} \def\cprime{$'$}

\end{document}